\documentclass[12pt]{article}
\usepackage{amsmath,amsfonts,amssymb,amsthm,fullpage,bm}
\usepackage{todonotes,stmaryrd,mathtools}
\usepackage[hidelinks]{hyperref}
\usepackage[nottoc]{tocbibind}
\usepackage{tocloft}

\newtheorem{thm}{Theorem}[section]
\newtheorem{lem}[thm]{Lemma}
\newtheorem{cor}[thm]{Corollary}
\newtheorem{prop}[thm]{Proposition}
\newtheorem{conj}[thm]{Conjecture}
\theoremstyle{definition}
\newtheorem{Def}[thm]{Definition}
\newtheorem*{open}{Open problem}
\newtheorem*{rem}{Remark}

\newcommand\C{\mathbb{C}}
\newcommand\R{\mathbb{R}}

\newcommand\T{\mathbb{T}}
\newcommand\F{\mathbb{F}}
\newcommand\sch{\mathcal{S}}
\newcommand{\slr}{\mathrm{SL}_2(\R)}

\newcommand\norm[1]{\|#1\|}

\newcommand\bab[1]{\left\lvert #1\right\rvert}

\newcommand\ab[1]{\lvert #1\rvert}
\newcommand\inn[1]{\left\langle #1\right\rangle}
\newcommand\inv{^{-1}}
\newcommand\up[1]{^{(#1)}}

\newcommand\wh[1]{\widehat{#1}}
\newcommand\ol[1]{\overline{#1}}

\newcommand\dd{\,\mathrm{d}}
\DeclareMathOperator\supp{supp}
\DeclareMathOperator\sgn{sgn}
\DeclareMathOperator\rank{rank}
\DeclareMathOperator\rksupp{rk-supp}
\DeclareMathOperator\minsupp{min-supp}

\DeclareMathOperator\tr{Tr}

\newcommand{\smat}[1]{\left(\begin{smallmatrix}#1\end{smallmatrix}\right)}

\title{The uncertainty principle: variations on a theme}
\author{Avi Wigderson\thanks{School of Mathematics, Institute for Advanced Study, Princeton, NJ 08540, USA. Email: \url{avi@ias.edu}. Research supported by NSF grant CCF-1900460} \and Yuval Wigderson\thanks{Department of Mathematics, Stanford University, Stanford, CA 94305, USA. Email: \url{yuvalwig@stanford.edu}. Research supported by NSF GRFP Grant DGE-1656518.}}

\begin{document}
\maketitle
\begin{abstract}
	We show how a number of well-known uncertainty principles for the Fourier transform, such as the Heisenberg uncertainty principle, the Donoho--Stark uncertainty principle, and Meshulam's non-abelian uncertainty principle, have little to do with the structure of the Fourier transform itself. Rather, all of these results follow from very weak properties of the Fourier transform (shared by numerous linear operators), namely that it is bounded as an operator $L^1 \to L^\infty$, and that it is unitary. Using a single, simple proof template, and only these (or weaker) properties, we obtain some new proofs and many generalizations of these basic uncertainty principles, to new operators and to new settings, in a completely unified way. Together with our general overview, this paper can also serve as a survey of the many facets of the phenomena known as uncertainty principles.
\end{abstract}

\tableofcontents
\section{Introduction}

\subsection{Background}

The phrase ``uncertainty principle'' refers to any of a wide class of theorems, all of which capture the idea that a non-zero function and its Fourier transform cannot both be 
``very localized''.  This phenomenon has been under intensive study for almost a century now, with  new results published continuously to this day. So while this introduction (and the paper itself) discusses some broad aspects of it, this is not a comprehensive survey. For more information on the history of the uncertainty principle, and for many other generalizations and variations, we refer the reader to the excellent survey of Folland and Sitaram \cite{FoSi}.

The study of uncertainty principles began with Heisenberg's seminal 1927 paper~\cite{Heisenberg}, with the corresponding mathematical formalism independently due to Kennard \cite{Kennard} and Weyl \cite{Weyl}. 
The original motivation for studying the uncertainty principle came from quantum mechanics\footnote{As some pairs of natural physical parameters, such as the position and momentum of a particle, can be viewed as such dual functions, the phrase ``uncertainty principle'' is meant to indicate that it is impossible to measure both to arbitrary precision.}, and thus most classical uncertainty principles deal with functions on $\R$ or $\R^n$. The first, so-called Heisenberg uncertainty principle, says that the \emph{variance} (appropriately defined, see Section \ref{sec:infinite-dim}) of a function and of its Fourier transform cannot both be small. Following Heisenberg's paper, many different notions of locality were studied. For example, it is a simple and well-known fact that if $f:\R \to \C$ has compact support, then $\hat f$ can be extended to a holomorphic function on $\C$, which in particular implies that $\hat f$ only vanishes on a discrete countable set, and so it is not possible for both $f$ and $\hat f$ to be compactly supported. This fact was generalized by Benedicks \cite{Benedicks} (and further extended by Amrein and Berthier \cite{AmBe}), who showed that it is not possible for both $f$ and $\hat f$ to have supports of finite measure. Another sort of uncertainty principle, dealing not with the sharp localization of a function, but rather with its decay at infinity, has also been widely studied. The first such result is due to Hardy \cite{Hardy}, who proved (roughly) that it is not possible for both $f$ and $\hat f$ to decay faster than $e^{-x^2}$. Yet another type of uncertainty principle is the logarithmic version conjectured by Hirschman \cite{Hirschman} and proven by Beckner~\cite{Beckner75} and independently Bia\l ynicki-Birula and Mycielski \cite{BiMy}, which deals with the Shannon entropies of a function and its Fourier transform, and which has connections to log-Sobolev and hypercontractive inequalities \cite{Beckner95}. 

In 1989, motivated by applications to signal 
processing\footnote{In this context, some similar theorems can be called ``certainty principles''. Here, this indicates that one can use the fact that one parameter is {\em not} localized to measure it well by random sampling.}, Donoho and Stark \cite{DoSt} initiated the study of a new type of uncertainty principle, which deals not with functions defined on $\R$, but rather with functions defined on finite groups. Many of the concepts discussed above, such as variance and decay at infinity, do not make sense when dealing with functions on a finite group. However, other measures of ``non-localization'', such as the size of the support of a function, are well-defined in this context, and the Donoho--Stark uncertainty principle deals with this measure. Specifically, they proved that if $G$ is a finite abelian group\footnote{Strictly speaking, Donoho and Stark only proved this theorem for cyclic groups, but it was quickly observed that the same result holds for all finite abelian groups.}, $\wh G$ is its dual group, $f:G \to \C$ is a non-zero function, and $\hat f:\wh G \to \C$ is its Fourier transform, then $\ab{\supp(f)}\ab{\supp(\hat f)} \geq \ab G$. They also proved a corresponding theorem for an appropriate notion of ``approximate support'' (see Section \ref{sec:approx-support} for more details). The work of Donoho and Stark led to a number of other uncertainty principles for finite groups. Three notable examples are Meshulam's extension \cite{Meshulam} of the Donoho--Stark theorem to arbitrary finite groups, Tao's strengthening \cite{Tao} of the Donoho--Stark theorem in case $G$ is a cyclic group of prime order, and the discrete entropic uncertainty principles of Dembo, Cover, and Thomas \cite{DeCoTh}, which generalize the aforementioned theorems of Hirschman \cite{Hirschman}, Beckner \cite{Beckner75} and Bia\l ynicki-Birula and Mycielski \cite{BiMy}.

Despite the fact that all uncertainty principles are intuitively similar, their proofs use a wide variety of techniques and a large number of special properties of the Fourier transform. Here is a sample of this variety. The standard proof of Heisenberg's uncertainty principle uses integration by parts, and the fact that the Fourier transform on $\R$ turns differentiation into multiplication by $x$. Benedicks's proof that a function and its Fourier transform cannot both have finite-measure supports uses the Poisson summation formula. The logarithmic uncertainty principle follows from differentiating a deep fact of real analysis, namely the sharp Hausdorff--Young inequality of Beckner \cite{Beckner75}. The original proof of the Donoho--Stark uncertainty principle uses the fact that the Fourier transform on a cyclic group $G$, viewed as a $\ab G \times \ab G$ matrix, is a Vandermonde matrix, and correspondingly certain submatrices of it can be shown to be non-singular. Tao's strengthening of this theorem also uses this Vandermonde structure, together with a result of Chebotar\"ev which says that in case $\ab G$ is prime, \emph{all} square submatrices of this Vandermonde matrix are non-singular. 
The proof of Donoho and Stark's approximate support inequality relates it to different norms of submatrices of the Fourier transform matrix. Finally, Meshulam's proof of the non-abelian
uncertainty principle uses linear-algebraic considerations in the group algebra $\C[G]$.

\subsection{The simple theme and its variations}
In this paper, we present a unified framework for proving many (but not all) of these uncertainty principles, together with various generalizations of them. The key observation throughout is that although the proofs mentioned above use a wide variety of analytic and algebraic properties that are particular to the Fourier transform, these results can also be proved using almost none of these properties. Instead, all of our results will follow from two very basic facts about the Fourier transform, namely that it is bounded as an operator $L^1 \to L^\infty$, and that it is unitary\footnote{In fact, a far weaker condition than unitarity is needed for our results, as will become clearer in the technical sections.}. Because the Fourier transform is by no means the only operator with these properties, we are able to extend many of these well-known uncertainty principles to many other operators.

This unified framework is, at its core, very simple. The $L^1 \to L^\infty$ boundedness of the Fourier transform gives an inequality relating $\norm{\hat f}_\infty$ to $\norm f_1$. Similarly, the $L^1 \to L^\infty$ boundedness of the inverse Fourier transform gives an analogous inequality relating $\norm f_\infty$ to $\norm{\hat f}_1$. Multiplying these two inequalities together yields our basic uncertainty principle, which has the form 
\[
	\frac{\norm f_1}{\norm f_\infty} \cdot \frac{\norm {\hat f}_1}{\norm {\hat f}_\infty} \geq C_0,
\]
for an appropriate constant $C_0$. Thus, in a sense, the ``measure of localization'' $H_0(g) = \frac{\norm g_1}{\norm g_\infty}$ is a primary one for us, and the uncertainty principle above, 
\begin{equation}\label{eq:L0}
    H_0(f) \cdot H_0({\hat f}) \geq C_0
\end{equation}
is the source of essentially all our uncertainty principles. Note that $H_0$ really is a yet another ``measure of localization'' of a function $g$, in that a function that is more ``spread out'' will have a larger $L^1$ norm than a more localized function, if both have the same $L^\infty$ norm. Here is how we will use this primary uncertainty principle.

Suppose we want to prove any uncertainty principle, for any potential ``measure of localization'' $H$ on functions, e.g.\ one of the form 
\begin{equation}\label{eq:L-up}
	H(f) \cdot H({\hat f}) \geq C.
\end{equation}
For example, our ``measure of localization'' $H$ might be the {\em variance} of (the square of) a function on the reals if we want to prove the Heisenberg uncertainty principle, or $H$ might be the {\em support size} of a finite-dimensional vector if we want to prove the Donoho--Stark uncertainty principle. 

We will derive (\ref{eq:L-up}) by first proving a universal\footnote{This bound is universal in the sense that no operator like the Fourier transform is involved in this inequality; it deals only with a single function $g$.} bound, relating the measure of localization $H$ to our primary one $H_0$, that holds for {\em every} function $g$. This reduction will typically take the form
\[ 
	H(g) \geq C' \cdot H_0(g) = C' \cdot \frac{\norm g_1}{\norm g_\infty} .
\]
Now this bound can be applied to both $f$ and $\hat f$ separately, which combined with our primary principle (\ref{eq:L0}) yields  (with $(C')^2 = C\cdot C_0$) the desired uncertainty principle (\ref{eq:L-up})\footnote{Variants of this idea will come in handy as well, e.g.\ proving that $H(g) \geq  H_0(g)^{C'}$, yielding $C=C_0^{C'}$.}.

Such an approach to proving simple uncertainty principles is by no means new; it goes back at least to work of Williams \cite{Williams} from 1979, who used it to prove a weak version of an approximate-support inequality for the Fourier transform on $\R$ (see Section \ref{sec:approx-support} for more details). Moreover, this approach to proving the Donoho--Stark uncertainty principle has apparently been independently rediscovered several times, e.g.\ \cite{Conrad,Ram}. However, we are not aware of any previous work on the wide applicability of this simple approach; indeed, we found no other applications besides the two mentioned above. Importantly, the separation of the two parts of the proof above is only implicit in these papers (after all, the whole proof in these cases is a few lines), and we believe that making this partition explicit and general, as presented above, is the source of its power. We note that though the second part of this two-part approach is often straightforward to prove, it occasionally becomes interesting and non-trivial; see for instance, Section \ref{sec:open-problems}.

In this paper, will see how this approach leads to a very different proof of the original Heisenberg uncertainty principle, in which the measure $H$ is the variance. We will then prove other versions of that classical principle, which yield uncertainty when $H$ captures higher moments than the variance. We will also see how it extends to uncertainty principles where $H$ captures several notions of approximate support and ``non-abelian'' support, as well as new uncertainty principles where $H$ captures the ratio between other pairs of norms (which in turn will be useful for other applications). Throughout the paper, we attempt to establish tightness of the bounds, at least up to constant factors. 

As mentioned, the {\em primary}  uncertainty principle uses a very basic property that far from special to the Fourier transform, but is shared by (and so applies to) many other operators in different discrete and continuous settings, which we call \emph{$k$-Hadamard operators}. 

Although the study of uncertainty principles is nearly a century old, it continues to be an active and vibrant field of study, with new results coming out regularly (e.g.\ \cite{CaMaOr,Erb,GoOlRa,Northington,Poria,QuRuSu}---all from the past 12 months!).
While many uncertainty principles are unlikely to fit into the simple framework above, we nonetheless hope that our technique will help develop this theory and find further applications.

\subsection{Outline of the paper}
The paper has two main parts, which have somewhat different natures. The first is on finite-dimensional uncertainty principles, and the second is on infinite-dimensional ones. Both parts have several sections, each with a different incarnation of the uncertainty principle. In most sections, we begin with a known result, which we then show how to reprove and generalize using our framework. We remark that most sections and subsections are independent of one another, and can be read in more or less any order. We therefore encourage the reader to focus on those sections they find most interesting.

Before delving into these, we start with the preliminary Section~\ref{sec:generalities}. In it, we first formally state and prove the (extremely simple) primary $L^1 \to L^\infty$ uncertainty principle, from which everything else will follow. This leads to a natural abstract definition of operators amenable to this proof, which we call $k$-Hadamard matrices; most theorems in the finite-dimensional section will be stated in these terms (and a similar notion will be developed for the infinite-dimensional section).

In Section \ref{sec:finite-dim}, we survey, reprove, and generalize finite-dimensional uncertainty principles, including the commutative and non-commutative support uncertainty principles of (respectively) Donoho--Stark and Meshulam, several old and new approximate support uncertainty principles, including those of Williams and Donoho--Stark, as well as some new uncertainty principles on norms. In Section \ref{sec:infinite-dim} we turn to infinite-dimensional vector spaces and the uncertainty principles one can prove there. These include support inequalities for the Fourier transform on topological groups and several variants and extensions of Heisenberg's uncertainty principle, in particular to higher moments. Moreover, these theorems apply to the general class of Linear Canonical Transforms (which vastly extend the Fourier transform). Finally, Appendix \ref{sec:appendix} collects the proofs of some technical theorems.

\section{The primary uncertainty principle and \texorpdfstring{$k$}{k}-Hadamard matrices}\label{sec:generalities}
As stated in the Introduction, the primary uncertainty principle that will yield all our other results is a theorem that lower-bounds the product of two $L^1$ norms by the product of two $L^\infty$ norms. In this section, we begin by stating this primary principle, as well as giving its (extremely simple) proof. We then define $k$-Hadamard matrices, which will be our main object of study in Section \ref{sec:finite-dim}, and whose definition is motivated from the statement of the primary uncertainty principle (roughly, the definition of $k$-Hadamard matrices is ``those matrices to which the primary uncertainty principle applies''). We end this section with several examples of $k$-Hadamard matrices, which show that such matrices arise naturally in many areas of mathematics, such as group theory, design theory, random matrix theory, coding theory, and discrete geometry.

\subsection{The primary uncertainty principle}
We begin by recalling the definition of an operator norm. 
Let $V,U$ be any two real or complex vector spaces, and let $\norm \cdot _{V}$ and $\norm \cdot _{U}$ be any norms on $V,U$, respectively. Let $A:V \to U$ a linear map.  Then the \emph{operator norm} of $A$ is 
\[
	\norm A_{V \to U} = \sup_{0 \neq v \in V} \frac{\norm{Av}_U}{\norm v_V} = \sup_{\substack{v \in V\\\norm v_V  = 1}} \norm{Av}_U.
\]
For $1 \leq p, q \leq \infty$, we will denote by $\norm A_{p \to q}$ the operator norm of $A$ when $\norm \cdot_V$ is the $L^p$ norm on $V$ and $\norm \cdot_U$ is the $L^q$ norm on $U$. 

With this notation, we can state our main theorem, which is the primary uncertainty principle that will underlie all our other results. We remark that this theorem, as stated below, is nearly tautological---our assumptions on the operators $A$ and $B$ are tailored to give the desired result by a one-line implication. Despite this simple nature, the strength of this theorem comes from the fact that many natural operators, such as the Fourier transform, satisfy these hypotheses.
\begin{thm}[Primary uncertainty principle]\label{thm:primary-up}
	Let $V,U$ be real or complex vector spaces, each equipped with two norms $\norm \cdot _1$ and $\norm \cdot_\infty$, and let $A:V \to U$ and $B:U \to V$ be linear operators. Suppose that $\norm A_{1 \to \infty} \leq 1$ and $\norm{B}_{1 \to \infty} \leq 1$. Suppose too that $\norm {BAv}_\infty \geq k \norm v_\infty$ for all $v \in V$, for some parameter $k>0$. Then for any $v \in V$,
	\[
		\norm v_1 \norm{Av}_1 \geq k \norm v_\infty \norm{Av}_\infty.
	\]
\end{thm}
\begin{proof}
	Since $\norm A_{1 \to \infty} \leq 1$, we have that 
	\[
		\norm {Av}_\infty \leq \norm v_1.
	\]
	Similarly, since $\norm B_{1 \to \infty} \leq 1$,
	\[
		\norm{BAv}_\infty \leq \norm{Av}_1.
	\]
	Multiplying these two inequalities together, we find that
	\[
		\norm v_1 \norm{Av}_1 \geq \norm{Av}_\infty \norm{BAv}_\infty \geq k \norm v_\infty \norm {Av}_\infty,
	\]
	as claimed.
\end{proof}
Note that Theorem \ref{thm:primary-up} holds regardless of the dimensions of $V$ and $U$ (so long as the $L^1$ and $L^\infty$ norms are well-defined on them), and in Section \ref{sec:infinite-dim}, we will use this primary uncertainty principle in infinite dimensions. But for the moment, let us focus on finite dimensions, in which case we take $\norm \cdot_1$ and $\norm \cdot_\infty$ to be the usual $L^1$ and $L^\infty$ norms on $\R^n$ or $\C^n$. When applying Theorem \ref{thm:primary-up}, we will usually take $B=A^*$. Note that the $1 \to \infty$ norm of a matrix is simply the maximum absolute value of the entries in the matrix, so $\norm A_{1 \to \infty} \leq 1$ if and only if all entries of $A$ are bounded by $1$ in absolute value. Moreover, if $B=A^*$, then\footnote{In fact, this holds in greater generality: as long as $\norm \cdot_1$ and $\norm \cdot_\infty$ are dual norms on any inner product spaces, we will have that $\norm A_{1 \to \infty} = \norm{A^*}_{1 \to \infty}$.} $\norm B_{1 \to \infty} = \norm A_{1 \to \infty}$. Thus, in this case, all we need to check in order to apply Theorem \ref{thm:primary-up} is that $\norm{A^*Av}_\infty \geq k \norm v_\infty$ for all $v \in V$. This motivates the following definition, of $k$-Hadamard matrices, which are defined essentially as ``those matrices that Theorem \ref{thm:primary-up} applies to''. We note that a similar definition was made by Dembo, Cover, and Thomas in \cite[Section IV.C]{DeCoTh}, and they state their discrete norm and entropy uncertainty principles in similar generality.
\begin{Def}
	Let $A \in \C^{m \times n}$ be a matrix and $k>0$. We say that $A$ is \emph{$k$-Hadamard} if every entry of $A$ has absolute value at most $1$ and $\norm{A^*Av}_{\infty}\geq k\norm v_\infty$ for all $v \in \C^n$. 
	Equivalently, $A$ is $k$-Hadamard if all its entries are bounded by $1$ in absolute value, $A^*A$ is invertible, and $\norm{(A^* A)\inv}_{\infty \to \infty} \leq 1/k$. 
\end{Def}
The next subsection consists of a large number of examples of $k$-Hadamard matrices which arise naturally in many areas of mathematics. Before proceeding, we end this subsection with three general observations. The first is the observation that the simplest way to ensure that\footnote{And indeed, to have that $\norm{A^*Av} = k \norm v$ for any norm whatsoever.} $\norm{A^* Av}_\infty \geq k \norm v_\infty$ is to assume that $A^* A = kI$. As we will see in the next section, many natural examples of $k$-Hadamard matrices have this stronger unitarity property.

Our second general observation is a rephrasing of Theorem \ref{thm:primary-up}, using the terminology of $k$-Hadamard matrices. Note that it is really identical to Theorem \ref{thm:primary-up}, but we state it separately for convenience.
\begin{thm}[Primary uncertainty principle, rephrased]\label{thm:hadamard-uncertainty}
	Let $A \in \C^{m \times n}$ be $k$-Hadamard. Then for any $v \in \C^n$, we have that
	\[
		\norm v_1 \norm{Av}_1 \geq k \norm v_\infty \norm{Av}_\infty.
	\]
\end{thm}

Thirdly, one can observe that the proof of Theorem \ref{thm:primary-up} never actually used any properties whatsoever of the (usual) $L^1$ and $L^\infty$ norms, and the same result holds (with appropriately modified assumptions) for any choice of norms. We chose above to state the primary uncertainty principle specifically for the $L^1$ and $L^\infty$ norms simply because it is that statement that will be used to prove all subsequent theorems. However, for completeness, we now state the most general version which holds for all norms. The proof is identical to that of Theorem \ref{thm:primary-up}, so we omit it.

\begin{thm}[Primary uncertainty principle, general version]\label{thm:general-case}
	Let $V,U$ be real or complex vector spaces, and let $A:V \to U$ and $B:U \to V$ be linear operators. Let $\norm \cdot_{V \up 1},\norm \cdot_{V \up 2}$ be two norms on $V$, and $\norm \cdot_{U \up 1},\norm \cdot_{U \up 2}$ two norms on $U$. Suppose that $\norm A_{V\up 1 \to U \up 2} \leq 1$ and $\norm B_{U \up 1 \to V \up 2} \leq 1$, and suppose that $\norm{BAv}_{V \up 2} \geq k \norm v_{V \up 2}$ for all $v \in V$ and some $k > 0$. Then for any $v \in V$,
	\[
		\norm v_{V\up 1} \norm{Av}_{U \up 1} \geq k \norm v_{V \up 2} \norm{Av}_{U \up 2}.
	\]
\end{thm}

\subsection{Examples of \texorpdfstring{$k$}{k}-Hadamard matrices}\label{subsec:k-hadamard-ex}
We end this section by collecting several classes of examples of $k$-Hadamard matrices. While this subsection may be skipped at first reading, its main point is to demonstrate the richness of operators for which uncertainty principles hold. As remarked above, many of these matrices actually satisfy the stronger property that $A^*A = kI$.

\begin{description}
	\item[Hadamard matrices] 
	Observe that if $A$ is a $k$-Hadamard $n\times n$ matrix, then $k \leq n$, and thus $n$-Hadamard matrices are best possible. One important class of $n$-Hadamard matrices are the ordinary\footnote{Whence our name for such matrices.} Hadamard matrices, which are $n\times n$ matrices $A$ with all entries in $\{-1,1\}$ and with $A^*A=nI$. There are many constructions of Hadamard matrices, notably Paley's constructions \cite{Paley} coming from quadratic residues in finite fields. Moreover, one can always take the tensor product of two Hadamard matrices and produce a new Hadamard matrix, which allows one to generate an infinite family of Hadamard matrices from a single example, such as the $2 \times 2$ matrix $\smat{1&1\\-1&1}$. Of course, these examples are nothing but the Fourier transform matrices over Hamming cubes. We remark that there are still many open questions about Hadamard matrices, most notably the so-called Hadamard conjecture, which asserts that $n \times n$ Hadamard matrices should exist for all $n$ divisible by $4$. 

	\item[Complex Hadamard matrices]
	However, $n$-Hadamard $n\times n$ matrices are more general than Hadamard matrices, because we do not insist that the entries be real. In fact, one can show that $n$-Hadamard matrices are precisely complex Hadamard matrices, namely matrices with entries on the unit circle $\{z \in \C:\ab z =1\}$ whose rows are orthogonal. There is a rich theory to these matrices, with connections to operator algebras, quantum information theory, and other areas of mathematics; for more, we refer to the survey~\cite{Banica}. 

	\item[The Fourier transform] 
	A very important class of complex Hadamard matrices consists of Fourier transform matrices: if $G$ is a finite abelian group, then we may normalize its Fourier transform matrix so that all entries have norm $1$, and the Fourier inversion formula precisely says that this matrix multiplied by its adjoint is $\ab GI$. Thus, Fourier transform matrices are $n$-Hadamard $n \times n$ matrices, where $n=\ab G$. More generally, quantum analogues of the Fourier transform can also be seen as $k$-Hadamard matrices; for more information, see \cite{JaJiLiReWu}.

	\item[Other explicit square matrices] 
	We do not insist that $k=n$ in our $n \times n$ Hadamard matrices. For $k<n$, special cases of $k$-Hadamard matrices with $A^*A=kI$ have been studied in the literature. For instance, such matrices with $k=n-1$, real entries, and zeros on the diagonal are called \emph{conference matrices}, and \emph{weighing matrices} are such matrices with $k<n$ and all entries in $\{-1,0,1\}$; both of these have been studied in connection with design theory. See \cite{CrKh} for more details.

	\item[Random matrices]
	For less structured examples of $k$-Hadamard matrices, let $M$ be a random $n\times n$ unitary (or orthogonal) matrix, i.e.\ a matrix sampled from the Haar measure on $\mathrm U(n)$ (or $\mathrm O(n)$). It is well-known that with high probability as $n \to \infty$, every entry of $M$ will have norm $O(\sqrt{\log n/n})$; see \cite[Theorem 8.1]{DoHu} for a simple proof, and \cite{Jiang} for far more precise results, including the determination of the correct constant hidden in the big-$O$. Thus, if we multiply $M$ by $c\sqrt{n/\log n}$ for an appropriate constant $c>0$, we will obtain with high probability a $k$-Hadamard matrix with $k=\Omega(n/\log n)$. This shows that an appropriately chosen random matrix will be $\Omega(n/\log n)$-Hadamard with high probability, which is best possible up to the logarithmic factor. 

	\item[Rectangular matrices and codes] 
	Recall that we do not require our $k$-Hadamard matrices to be square, which corresponds to not insisting that $V$ and $U$ have the same dimension in Theorem \ref{thm:general-case}. If $A$ is an $m \times n$ $k$-Hadamard matrix, then $k \leq m$. One example of a non-square matrix attaining this bound is the $2^n \times n$ matrix $S$ whose rows consist of all vectors in $\{-1,1\}^n$. Then distinct columns of $S$ are orthogonal because they disagree in exactly $2^{n-1}$ coordinates, which implies that $S^*S=2^n I$, and so $S$ is $2^n$-Hadamard. 
	
	Note that the columns of $S$ are simply the codewords of the Hadamard code, viewed as vectors in $\{-1,1\}^{2^n}$ rather than in $\{0,1\}^{2^n}$. A similar construction works for all binary codes with an appropriate minimum distance. If we view the codewords as vectors with $\pm 1$ entries and form a matrix $S$ whose columns are these codewords, then the minimum distance condition will imply that the columns are nearly orthogonal, and thus that the diagonal entries of $S^*S$ will be much larger than the off-diagonal entries. Thus, $S$ will be $k$-Hadamard for a value of $k$ depending on the minimum distance of the code.

	\item[Incidence matrices of finite geometries] 
	If $q$ is a prime power, let $\mathrm{PG}(2,q)$ be the projective plane over the field $\F_q$. Then setting $n=q^2+q+1$, we can let $A$ be the $n\times n$ incidence matrix of points and lines in $\mathrm{PG}(2,q)$, namely the matrix whose rows and columns are indexed by the points and lines, respectively, of $\mathrm{PG}(2,q)$, and whose $(p,\ell)$ entry is $1$ if $p \in \ell$, and $0$ otherwise. Then $A$ certainly has all its entries bounded by $1$. Moreover, each column has exactly $q+1$ ones, and distinct columns have inner product $1$, since any two lines intersect at exactly one point. This implies that $A^*A = qI+J$, where $J$ is the all-ones matrix. It is not too hard to see from this that $\norm{A^* A v}_\infty \geq \frac{q}{2}\norm v_\infty$ for any $v \in \C^n$, which implies that $A$ is a $k$-Hadamard $n\times n$ matrix, where $k \geq q/2 =\Theta(\sqrt n)$. 

	Similarly, one can consider the $d$-dimensional projective space $\mathrm{PG}(d,q)$ over $\F_q$, and form the incidence matrix of $a$-flats and $b$-flats, for any $0\leq a<b<d$. It will be a (not necessarily square) $k$-Hadamard matrix, for some value of $k$ depending on $a,b,$ and $d$.
\end{description}

\section{Finite-dimensional uncertainty principles}\label{sec:finite-dim}
In this section, we show how to use our general uncertainty principle for $k$-Hadamard matrices, Theorem \ref{thm:hadamard-uncertainty}, to prove a number of uncertainty principles in finite dimensions. We start with the basic support-size 
uncertainty principle of Donoho and Stark, and then move on to Meshulam's generalization of it to arbitrary finite groups. We next proceed to prove several uncertainty principles for various notions of approximate support, and conclude with a collection of uncertainty principles for ratios of other norms, which will be useful for us later when we prove the Heisenberg uncertainty principle.

Most of our results in this section generalize known theorems about the Fourier transform on finite groups. However, as we demonstrate below, they do not actually need any of the algebraic structure of the Fourier transform (or even of an underlying group), and instead all follow from the fact that Fourier transform matrices are $k$-Hadamard.

\subsection{The Donoho--Stark support-size uncertainty principle}
For a vector $v \in \C^n$, let $\supp(v)$ be its \emph{support}, namely the set of coordinates $i$ where $v_i \neq 0$. Similarly, if $f:G \to \C$ is a function on a finite group, we denote by $\supp(f)$ the set of $x \in G$ for which $f(x) \neq 0$. Recall that if $G$ is a finite abelian group, we denote by $\wh G$ the \emph{dual group}, which consists of all homomorphisms from $G$ to the circle group $\T=\{z \in \C:\ab z = 1\}$. $\wh G$ forms an abelian group under pointwise multiplication, and it is in fact (non-canonically) isomorphic to $G$. We define the Fourier transform $\hat f:\wh G \to \C$ of a function $f:G \to \C$ by $\hat f(\chi) = \sum_{x \in G} f(x) \ol{\chi(x)}$. The basic uncertainty principle for the Fourier transform on finite abelian groups is the following theorem of Donoho and Stark.
\begin{thm}[Donoho--Stark \cite{DoSt}]\label{thm:donoho-stark}
	Let $G$ be a finite abelian group. If $f:G \to \C$ is a non-zero function and $\hat f:\wh G \to \C$ denotes its Fourier transform, then
	\[
		\ab{\supp(f)} \ab{\supp(\hat f)} \geq \ab G.
	\]
\end{thm}
Our first finite-dimensional result is an extension of Theorem \ref{thm:donoho-stark} to arbitrary $k$-Hadamard matrices.
\begin{thm}[Support-size uncertainty principle]\label{thm:hadamard-ds}
	Let $A \in \C^{m \times n}$ be a $k$-Hadamard matrix. Then for any non-zero $v \in \C^n$,
	\[
		\ab{\supp(v)}\ab{\supp(Av)} \geq k.
	\]
\end{thm}
\begin{proof}
	This is the first demonstration of the principle articulated in the Introduction. We already have, from Theorem \ref{thm:hadamard-uncertainty}, that for any non-zero $v$, $\norm v_1 \norm{Av}_1 \geq k \norm v_\infty \norm{Av}_\infty$.  Thus, all we need is to bound the support-size of a function by the ratio of its norms, which is obvious: for any vector $u$,
	\[
		\norm u_1 = \sum_{i=1}^{n} \ab{u_i} =\sum_{i \in \supp(u)} \ab{u_i} \leq \ab{\supp(u)} \norm u_\infty.
	\]
	Applying this bound to both $v$ and $Av$, we obtain the result.
\end{proof}

In this proof, the measure of localization we wished to study was the support of a vector. The uncertainty principle for this measure follows from the primary one (on the ratio of norms) via an inequality that holds for all vectors (bounding it by that ratio). This is an instance of the basic framework discussed in the Introduction, which will recur throughout. 

\begin{rem}
	We remark that in general, the bound in Theorem \ref{thm:donoho-stark} (and thus also in Theorem \ref{thm:hadamard-ds}) is tight. For instance, if $f$ is the indicator function of some subgroup $H \subseteq G$, then $\hat f$ will be a constant multiple of the indicator function of the dual subgroup $H^\perp \subseteq \wh G$, and we have that $\ab H \ab{H^\perp} = \ab G$. Thus, $\ab{\supp(f)} \ab{\supp(\hat f)}=\ab G$.
\end{rem}

\subsection{Support-size uncertainty principles for general finite groups}\label{subsec:non-abelian-gps}
In this section, we show how our general framework can be used to extend the Donoho--Stark support-size uncertainty principle to the Fourier transform over arbitrary finite groups, abelian or non-abelian. Such an extension was already proved by Meshulam~\cite{Meshulam}, for a {\em linear-algebraic} notion of support-size. Here we propose a natural {\em combinatorial} notion of support\footnote{In his paper, Meshulam also defines a certain combinatorial measure of support size, which (as he points out) is much weaker than his linear-algebraic one.}, and prove an uncertainty principle for it within our framework. Further, we prove that these two uncertainty principles are almost equivalent:  they  are identical  for a certain class of functions, and are always equivalent up to a factor of $4$. We note that both notions of support-size are natural and both extend the abelian case. Finally, at the end of the section, we provide another, new uncertainty principle of norms for general groups, proved by Greg Kuperberg, which provides a different proof Meshulam's theorem using our framework.

To facilitate a natural combinatorial definition of support, we embed both the
``time domain'' (namely, functions on the group), and the ``Fourier domain''
(namely, their image under the Fourier transform) as sub-algebras of
the matrix ring $\C^{n\times n}$, where $n=\ab G$. Then the notion of support becomes the standard one, namely the set of non-zero entries of these matrices. This embedding does
much more. It gives as well a natural definition of norms
(treating these matrices as vectors), and accommodates a description of
the Fourier transform as a $k$-Hadamard operator. These yield a proof of
the our support-size uncertainty inequality that is almost identical to
the one in the abelian case.

\subsubsection{Preliminaries: the Fourier transform in general finite groups}

We now recall the basic notions of the Fourier transform of general
finite groups (aka their representation theory) using the embedding above,
which also affords a definition of the inverse Fourier transform which
looks nearly identical to the abelian case. We refer the
reader to the comprehensive text \cite{Serre} on the representation theory of finite
groups for a standard exposition of these concepts.

Let $G$ be an arbitrary finite group of order $n$, and let $\C[G]$ denote its group algebra. We embed $\C[G]$ as a sub-algebra of $\C^{n\times n}$ as follows. Given an element $f = \sum_{x \in G}f(x) x$, let $T_f$ denote left-multiplication by $f$ in $\C[G]$. Then $T_f$ is a linear map $\C[G] \to \C[G]$. Moreover, since $\C[G]$ is equipped with a standard basis, namely the basis of delta functions on $G$, we can represent $T_f$ as an $n \times n$ matrix, and it is straightforward to see that both addition and multiplication of matrices corresponds to addition and multiplication in $\C[G]$. So we henceforth think of $\C[G]$ as the subspace of $\C^{n\times n}$ consisting of all matrices $T_f$.

If $G$ were abelian, then conjugating by the Fourier transform matrix would simultaneously diagonalize all $T_f$, with the diagonal entries precisely being the values of $\hat f$. If $G$ is non-abelian, then such a complete simultaneous diagonalization is impossible, but we can get maximal one possible; namely, conjugating by an appropriate matrix, which we also call the Fourier transform, turns each $T_f$ into a block-diagonal matrix with specified block sizes, uniformly for all $f$, as follows.

Let $\rho_1,\ldots,\rho_t$ be the irreducible representations of $G$ over $\C$, i.e.\ each $\rho_i$ is a homomorphism $G \to GL(W_i)$, where $W_i$ is a vector space over $\C$ of dimension $d_i$. We may assume  that $\rho_1,\ldots,\rho_t$ are unitary representations, meaning that $\rho_i(x)$ is a unitary transformation on $W_i$ for all $x \in G$. Recall that $n = d_1^2+\dotsb+d_t^2$. Then we define the Fourier transform as follows.
\begin{Def}[The Fourier transform]
	Given a function $f:G \to \C$, its \emph{Fourier transform} is defined by $\hat f(\rho_i)=\sum_{x \in G}f(x) \rho_i(x)$, so that $\hat f(\rho_i)$ is a linear transformation $W_i \to W_i$.
\end{Def}

We also henceforth fix an orthonormal basis $E_i$ of each $W_i$, and everything that follows will implicitly depend on these choices of bases. In particular, we may now think of the linear maps $\rho_i(x)$ and $\hat f(\rho_i)$ as a $d_i \times d_i$ matrices, represented in the basis $E_i$. We remark that, as above, one can define the Fourier transform without reference to any bases, but that everything we do from now on, such as defining the support and its size, will need these bases.

To define the Fourier transform matrix, we describe its columns (first, up to scaling): these are the so-called \emph{matrix entry} vectors. For indices $i \in [t]$ and $j,k \in [d_i]$, we define the matrix entry vector $c(i;j,k)\in \C^n$ as follows. It is a vector whose coordinates are indexed by elements of $G$, and whose $x$ coordinate is the $(j,k)$ entry of the matrix $\rho_i(x)$;
observe that in the abelian case these vectors are simply the $n$ characters of $G$. A simple consequence of Schur's lemma is that these vectors are orthogonal; see \cite[Corollaries 2--3]{Serre} for a proof. 
\begin{prop}[Orthogonality of matrix entries]\label{prop:matrix-entries-orth}
	We have
	\[
		\inn{c(i;j,k),c(i';j',k')} =
		\begin{cases}
			n/d_i&\text{if }i=i',j=j',k=k'\\
			0&\text{otherwise.}
		\end{cases}
	\]
\end{prop}
We can now formally define the Fourier transform matrix and establish its basic properties.
\begin{Def}[Fourier transform matrix]
	Let $F$ be the $n\times n$ matrix whose rows are indexed by $G$ and whose columns are indexed by tuples $(i;j,k)$ in lexicographic order, and whose $(i;j,k)$ column is the vector $\sqrt{d_i}c(i;j,k)$. We call $F$ the \emph{Fourier transform matrix}. 
\end{Def}
Observe that Proposition \ref{prop:matrix-entries-orth} implies that $F^*F = FF^* = nI$. Moreover, the key diagonalization property of $F$ mentioned above is that for any function $f:G \to \C$, we have that $FT_fF^*$ is a block-diagonal matrix, whose blocks are just $d_i$ copies of the matrices\footnote{It is a little strange to have the blocks of $\wh{T_f}$ be $n \hat f(\rho_i)$, rather than simply $\hat f(\rho_i)$. Of course, we could have normalized $F$ differently, so as to avoid this factor of $n$. However, we chose not to do this to be consistent with our earlier normalization of $k$-Hadamard matrices.} $n\hat f(\rho_i)$. 

Thus, we think of the Fourier transform as simply a change of basis (and a dilation) on the matrix space $\C^{n\times n}$. Recall that we had already embedded $\C[G]$ in this space by mapping $f \in \C[G]$ to the matrix $T_f$. We think of its Fourier transform as the block-diagonal matrix $\wh{T_f}\coloneqq F T_f F^*$. Moreover, we think of the subspace of $\C^{n \times n}$ consisting of all block-diagonal matrices with $d_i$ identical blocks of size $d_i \times d_i$ as the ``Fourier subspace''. Then the change of basis given by $F$ precisely maps the subspace corresponding to $\C[G]$ to this Fourier subspace. Note that if $G$ is abelian, then each $W_i$ is one-dimensional, and we find that $\wh{T_f}$ is simply a diagonal matrix whose diagonal entries are the values of $\hat f$.

\subsubsection{Notions of support for \texorpdfstring{$\hat f$}{f-hat}}

With this setup, there is a clear candidate for the support of $\hat f$. Namely, we define $\supp(\hat f)$ to simply be the set of non-zero entries of the matrix $\wh{T_f}$. Note that since the block $\hat f(\rho_i)$ appears $d_i$ times in $\wh{T_f}$, we have that
\[
	\ab{\supp(\hat f)} = \sum_{i=1}^t d_i \ab{\supp(\hat f(\rho_i))},
\]
where $\supp(\hat f(\rho_i))$ denotes the set of non-zero entries of the matrix $\hat f(\rho_i)$. Recall that this matrix depended on the choice of the bases $E_i$, so we also make the following definition.
\begin{Def}
	The \emph{minimum support-size} of $\hat f$ is
	\[
		\ab{\minsupp(\hat f)} = \min_{E_1,\ldots,E_t} \ab{\supp(\hat f)},
	\]
	where the minimum is over all choices of orthonormal bases $E_1,\ldots,E_t$ for $W_1,\ldots,W_t$.
\end{Def}
Thus, the minimum support-size of $\hat f$ is simply its support size in its most efficient representation. We note that if $G$ is abelian, then all $W_i$ are one-dimensional, and in particular the choice of basis affects nothing. So if $G$ is abelian, then both $\ab{\supp(\hat f)}$ and $\ab{\minsupp(\hat f)}$ simply recover our earlier notion of the support-size of $\hat f$.

Meshulam proposed an alternative notion for the support-size of $\hat f$, which we call the rank-support, and which is defined as follows.
\begin{Def}[Meshulam \cite{Meshulam}]
	Given $f:G \to \C$, the \emph{rank-support} of $\hat f$ is $\rksupp(\hat f) = \rank T_f$.
\end{Def}
We note that, with this definition, it is not at all clear how $\rksupp(\hat f)$ even depends on $\hat f$, let alone in what sense it can be thought of as a support-size. However, since similar matrices have the same rank, we see that $\rank T_f = \rank \wh{T_f}$, and since $\wh{T_f}$ is a block-diagonal matrix, we see that
\[
	\rank \wh{T_f} = \sum_{i=1}^t d_i \rank(\hat f(\rho_i)).
\]
In particular, this connection shows that if $G$ is abelian, then we also get that $\rksupp(\hat f) = \ab{\supp(\hat f)}$. Meshulam's definition of rank-support has some advantages over our definition of minimum support-size; importantly, the rank-support does not depend on the choices of bases $E_1,\ldots,E_t$. However, it is not obviously related to any notion of {\em support} of the Fourier transform---instead, it jumps directly to a notion of its {\em size}. In contrast, we offer, for each basis, a notion of support of $\hat f$, namely the set of non-zero entries in $\wh{T_f}$, and then we pick the smallest possible (i.e.\ in the most ``efficient'' basis) to define its size. As mentioned, both definitions agree with $\ab{\supp(\hat f)}$ for abelian groups $G$, so both can be considered reasonable notions of support-size. The two notions can be related as follows, where we say that a function $f$ is \emph{Hermitian} if $f(x)=\ol{f(x\inv)}$ for all $x \in G$.
\begin{lem}\label{lem:rksupp-and-minsupp}
	For any function $f:G \to \C$, we have that $\rksupp(\hat f) \leq \ab{\minsupp(\hat f)}$. Moreover, if $f$ is a {Hermitian} function, then $\rksupp(\hat f) = \ab{\minsupp(\hat f)}$. 
\end{lem}
\begin{proof}
	If $\wh{T_f}$ has $s$ non-zero entries, then in particular it has at most $s$ non-zero columns, which implies that $\rank \wh{T_f} \leq \ab{\supp(\hat f)}$ for any bases $E_1,\ldots,E_t$. This implies the first inequality by minimizing over these bases.

	For the second, suppose that $f$ is Hermitian. This implies that each $\hat f(\rho_i)$ is Hermitian, since
	\[
		\left(\hat f(\rho_i)\right)^* = \sum_{x \in G} \ol{f(x)} \left(\rho_i(x)\right)^*=\sum_{x \in G} \ol{f(x)} \rho_i(x\inv) = \sum_{x \in G} f(x\inv) \rho_i(x\inv) = \hat f(\rho_i).
	\]
	Thus, there exists an orthonormal basis $E_i$ for $W_i$ in which $\hat f(\rho_i)$ is a diagonal matrix. Using such a basis for each $W_i$, we see that $\wh{T_f}$ is diagonal, at which point its rank precisely equals the number of non-zero diagonal entries. This proves the reverse inequality for Hermitian $f$.
\end{proof}

\subsubsection{Uncertainty principles for the min-support and the rank-support}
One other connection between the rank-support and the minimum support-size is that both of them satisfy an uncertainty principle like that of Donoho and Stark. For the rank-support, this was proven by Meshulam.
\begin{thm}[Meshulam \cite{Meshulam}]\label{thm:meshulam}
	For any finite group $G$ and any $f:G \to \C$,
	\[
		\ab{\supp(f)}\rksupp(\hat f) \geq \ab G.
	\]
\end{thm}
For the minimum support-size, this is main result of this section.
\begin{thm}\label{thm:non-abelian-up}
	For any finite group $G$ and any $f:G \to \C$,
	\[
		\ab{\supp(f)}\ab{\minsupp(\hat f)} \geq \ab G.
	\]
\end{thm}
From Lemma \ref{lem:rksupp-and-minsupp}, we see that Meshulam's Theorem \ref{thm:meshulam} implies our Theorem \ref{thm:non-abelian-up}. Moreover, for Hermitian functions $f$, the two theorems are precisely equivalent. Finally, we can prove a reverse implication, up to a factor of $4$.
\begin{lem}
	Theorem \ref{thm:non-abelian-up} implies that for any function $f:G \to \C$,
	\[
		\ab{\supp(f)}\rksupp(\hat f) \geq \frac{\ab G}4.
	\]
\end{lem}
\begin{proof}
	Consider the function $g:G \to \C$ defined by $g(x) = f(x)+\ol{f(x\inv)}$. Then $g$ is Hermitian, so by Lemma \ref{lem:rksupp-and-minsupp} we have that $\rksupp(\hat g)= \ab{\minsupp(\hat g)}$. 
	As both the rank of matrices and the support of vectors are subadditive, we have that
	\[
		\ab{\supp(g)} \leq 2 \ab{\supp(f)} \qquad \text{ and }\qquad \rksupp(\hat g) \leq 2 \rksupp(\hat f).
	\]
	Putting this all together, we find that
	\[
		\ab{\supp(f)} \rksupp(\hat f) \geq \frac 14 \ab{\supp(g)} \rksupp(\hat g) = \frac 14 \ab{\supp(g)} \ab{\minsupp(\hat g)} \geq \frac{\ab G}{4}. \qedhere
	\]
\end{proof}
\begin{rem}
	The bound in Meshulam's Theorem \ref{thm:meshulam} is tight if $f$ is the indicator function of some subgroup of $G$, as observed in \cite{Meshulam}. As such an indicator function is Hermitian, this also shows that the bound in Theorem \ref{thm:non-abelian-up} is in general best possible. 
\end{rem}

Our proof of Theorem \ref{thm:non-abelian-up} more or less follows the abelian case, namely it's a direct application of our general framework. We define linear operators $A,B:\C^{n \times n} \to \C^{n \times n}$ by\footnote{We use the notation $\circ$ to denote the action of $A$ and $B$ to avoid confusion with the notation for matrix multiplication.} 
$A\circ M = F MF^*$ and $B\circ M = F^* M F$, for any $M \in \C^{n\times n}$. Note that $A\circ T_f = \wh{T_f}$. The main properties of these operators are captured in the following lemma, which simply says that $A$ acts as an $n^2$-Hadamard operator on the subspace $\C[G] \subset \C^{n \times n}$.
\begin{lem}\label{lem:non-ab-fourier-props}
	Let $\C[G] \cong V \subset \C^{n\times n}$ be the subspace consisting of the matrices $T_f$, and let $U \subset \C^{n \times n}$ be the Fourier subspace consisting of all block-diagonal matrices with $d_i$ identical blocks of size $d_i \times d_i$. Then the following hold.
	\begin{enumerate}
		\item For any $M \in V$, we have that $B\circ (A \circ M) = n^2 M$.
		\item For any $M \in V$, we have that $\norm{A\circ M}_\infty \leq \norm M_1$
		\item For any $N \in U$, we have that $\norm{B\circ N}_\infty \leq \norm N_1$.
	\end{enumerate}
\end{lem}
The proof is straightforward, and we defer it to Appendix \ref{sec:appendix}. However, with this lemma in hand, we can prove Theorem \ref{thm:non-abelian-up}.
\begin{proof}[Proof of Theorem \ref{thm:non-abelian-up}]
	We begin by fixing bases $E_1,\ldots,E_t$ that are minimizers in the definition of $\ab{\minsupp(\hat f)}$. By applying the support-size uncertainty principle for $k$-Hadamard matrices, Theorem \ref{thm:primary-up}, to the operators $A,B$ as above, we see that
	\[
		\ab{\supp(T_f)} \ab{\supp(A\circ T_f)} \geq n^2.
	\]
	Recall that $A\circ T_f$ is simply $\wh{T_f}$, so the second term is simply 
	$\ab{\minsupp(\hat f)}$. For the first term, observe that each column of $T_f$ is simply a permutation of the values that $f$ takes. This implies that $\ab{\supp(T_f)} = n \ab{\supp(f)}$, as every non-zero value of $f$ is repeated exactly $n$ times in $T_f$. Thus, dividing by $n$ gives the claimed bound.
\end{proof}

\subsubsection{Kuperberg's proof of Meshulam's uncertainty principle}
After reading a draft of this paper, Greg Kuperberg (personal communication) discovered a new norm uncertainty principle for the Fourier transform over non-abelian groups, which can be proved using our framework and which implies Meshulam's Theorem \ref{thm:meshulam} as a simple corollary. To state Kuperberg's theorem, we first need to define the Schatten norms of a matrix.
\begin{Def}[Schatten norms]
	Let $M \in \C^{n \times n}$ be a matrix and $p \in [1,\infty]$ be a parameter. The \emph{Schatten $p$-norm} of $M$, denoted $\norm M_p\up S$, is defined by
	\[
		\norm M_p\up S = \tr \left( (M^* M)^{p/2} \right) ^{1/p}.
	\]
	Equivalently, if $\sigma = (\sigma_1,\ldots,\sigma_n)$ is the vector of singular values of $M$, then the Schatten $p$-norm of $M$ is simply the ordinary $L^p$ norm of $\sigma$, i.e.\ $\norm M_p \up S = \norm \sigma_p$. 
\end{Def}
The Schatten norms are invariant under left- or right-multiplication by unitary matrices, so $\norm{T_f}\up S_p = \frac 1n \norm{\wh{T_f}}\up S_p$, with the factor of $n$ coming from our normalization of $F$ so that $\frac{1}{\sqrt n}F$ is unitary. Moreover, since $\wh{T_f}$ is a block-diagonal matrix with $d_i$ blocks of $n\hat f{\rho_i}$, we have that
\begin{equation}\label{eq:schatten-formulas}
	\norm{T_f}_1 \up S = \frac 1n\norm{\wh{T_f}}_1 \up S = \sum_{i=1}^t d_i \norm{\hat f(\rho_i)}_1 \up S \quad\text{and}\quad \norm{T_f}_\infty \up S = \frac 1n \norm{\wh{T_f}}_\infty \up S = \max_{i \in [t]} \norm{\hat f(\rho_i)}_\infty \up S.
\end{equation}
With this definition, we can state Kuperberg's norm uncertainty principle for the Fourier transform over finite groups. We state it for the Schatten norms of $\wh{T_f}$, but by (\ref{eq:schatten-formulas}), we could just as well replace $\wh{T_f}$ by $T_f$ in the following theorem.
\begin{thm}[Kuperberg]\label{thm:kuperberg}
	Let $G$ be a group of order $n$ and $f:G \to \C$ a non-zero function. We have that
	\[
		\frac{\norm f_1}{\norm f_\infty} \cdot \frac{\norm{\wh{T_f}}_1 \up S}{\norm{\wh{T_f}}_\infty \up S} \geq n.
	\]
\end{thm}
Meshulam's Theorem \ref{thm:meshulam} is a simple corollary of this theorem.
\begin{proof}[Proof of Theorem \ref{thm:meshulam}]
	We already know that $\ab{\supp(f)} \geq \norm f_1/\norm f_\infty$. Additionally, it is well-known that for any matrix $M$,
	\[
		\rank M \geq \frac{\norm M_1 \up S}{\norm M_\infty \up S}.
	\]
	This can be seen from the definition of $\norm M_p \up S$ as the $L^p$ norm of the vector $\sigma$ of singular values of $M$. Indeed, the rank of $M$ is simply the number of non-zero singular values, i.e.\ $\rank M = \ab{\supp(\sigma)}$, from which the above inequality follows. This shows, using Theorem \ref{thm:kuperberg}, that
	\[
		\ab{\supp(f)} \rksupp(\hat f) = \ab{\supp(f)} \rank{\wh{T_f}} \geq \frac{\norm f_1}{\norm f_\infty} \cdot \frac{\norm{\wh{T_f}}_1 \up S}{\norm{\wh{T_f}}_\infty \up S} \geq n.\qedhere
	\]
\end{proof}
So to finish the proof, we need to prove Kuperberg's Theorem \ref{thm:kuperberg}, whose proof is another application of our general framework.
\begin{proof}[Proof of Theorem \ref{thm:kuperberg}]
	If we think of $f \in \C[G]$ as a row vector, then we can think of $F$ as a linear operator $\C[G] \to U$, which sends $f$ to $\frac 1n\wh{T_f}$; note the additional factor of $\frac 1n$, coming from our earlier normalization of $\wh{T_f} = FT_fF^*$. We first claim that $\norm{\frac 1n\wh{T_f}}_\infty \up S \leq \norm f_1$, i.e.\ that $F$ has norm $1$ as an operator from $\norm \cdot_1$ to $\norm \cdot_\infty \up S$.
	By convexity of the Schatten-$\infty$ norm, it suffices to check this on the extreme points of the $L^1$ unit ball, i.e.\ for a delta function $f=\delta_x$, the function that takes value $1$ at $x \in G$ and $0$ on all other elements of $G$. But $T_{\delta_x}$ is simply a permutation matrix, so all of its singular values are $1$, implying that $\norm{T_{\delta_x}}_1 \up S =1$. This in turn implies that $\norm{\wh{T_{\delta_x}}}_1 \up S= n$, by (\ref{eq:schatten-formulas}).

	Recall that the Schatten-$1$ and Schatten-$\infty$ norms are dual on the matrix space $\C^{d_i \times d_i}$, which implies that they are dual on $U$ by the formulas in (\ref{eq:schatten-formulas}). Since the $L^1$ and $L^\infty$ norms on $\C[G]$ are also dual, the above also implies that $F^*$ has norm $1$ as an operator from $\norm \cdot_1 \up S$ to $\norm \cdot_\infty$. Finally, we already know that $F^*F = nI$, so we conclude by the primary uncertainty principle, Theorem \ref{thm:primary-up}, that
	\[
		\norm f_1 \norm{\wh{T_f}}_1 \up S \geq n \norm f_\infty \norm{\wh{T_f}}_\infty \up S.\qedhere
	\]
\end{proof}

\subsection{Uncertainty principles for notions of approximate support}\label{sec:approx-support}
The support-size uncertainty principle of Theorem \ref{thm:hadamard-ds} is rather weak, in the sense that the support size of a vector is a very fragile measure: coordinates with arbitrarily small non-zero values contribute to it. Stronger versions of this theorem, in which one considers instead the ``essential support'', namely the support of a vector after deleting such tiny entries\footnote{More precisely, deleting a small fraction of the total mass in some norm.}, are much more robust. Such versions were sought first by Williams \cite{Williams} in the continuous setting, and by Donoho and Stark \cite{DoSt} in the discrete setting.  It turns out that using our approach it is easy to extend Theorem \ref{thm:hadamard-ds} to such a robust form for the $L^1$ norm, but not for $L^2$ (although Donoho and Stark's original $L^2$ proof does generalize to $k$-Hadamard matrices with $A^*A=kI$). We will describe both and compare them. 

We start with some notation. If $v \in \C^n$ is a vector and $T \subseteq [n]$ is a set of coordinates, we denote by $v[T]$ the vector in $\C^{\ab T}$ obtained by restricting $v$ to the coordinates in $T$. We use $T^{\mathsf{c}}$ to denote the complement of $T$ in $[n]$. 
\begin{Def}
	Let $\varepsilon \in [0,1]$ and $p \in [1,\infty]$. For a vector $v \in \C^n$ and a set $T \subseteq [n]$, we say that $v$ is \emph{$(p,\varepsilon)$-supported} on $T$ if $\norm{v[T^{\mathsf{c}}]}_p \leq \varepsilon \norm v_p$.

	We also define the \emph{$(p,\varepsilon)$-support size} of $v$ to be
	\[
		\ab{\supp_\varepsilon^p(v)} = \min \{\ab T: T \subseteq [n], v\text{ is }(p,\varepsilon)\text{-supported on }T\}.
	\]
\end{Def}
\begin{rem}
	In general, there may not exist a unique minimum-sized set $T$ in the definition of $\ab{\supp_\varepsilon^p(v)}$, so the set ``$\supp_\varepsilon^p(v)$'' is not well-defined. However, we will often abuse notation and nevertheless write $\supp_\varepsilon^p (v)$ to denote an arbitrary set $T$ achieving the minimum in the definition of $\ab{\supp_\varepsilon^p(v)}$. 
\end{rem}

The basic uncertainty principle concerning approximate supports was also given by Donoho and Stark, who proved the following.
\begin{thm}[Donoho--Stark \cite{DoSt}]\label{thm:ds-approx-support}
	Let $G$ be a finite abelian group and $f:G \to \C$ a non-zero function. For any $\varepsilon,\eta \in [0,1]$, we have that
	\[
		\ab{\supp_\varepsilon^2 (f)} \ab{\supp_\eta^2(\hat f)} \geq  \ab G (1- \varepsilon- \eta)^2.
	\]
\end{thm}
If one attempts to apply our basic framework to prove an uncertainty principle for approximate supports, one is naturally led to the following result. The analogous theorem for the Fourier transform on $\R$ was proven by Williams \cite{Williams}, in what we believe is the earliest application of this paper's approach to uncertainty principles.
\begin{thm}[$L^1$-approximate support uncertainty principle]\label{thm:approx-support}
	Let $A \in \C^{m \times n}$ be a $k$-Hadamard matrix and let $v \in \C^n$ be a non-zero vector. For any $\varepsilon,\eta \in [0,1]$, we have that
	\[
		\ab{\supp_\varepsilon^1(v)} \ab{\supp_\eta^1(Av)} \geq k (1- \varepsilon)(1- \eta)\geq k(1- \varepsilon -\eta).
	\]
\end{thm}
\begin{proof}
	The second inequality follows from the first, so it suffices to prove the first. We may assume that $\varepsilon,\eta<1$. The primary uncertainty principle, Theorem \ref{thm:hadamard-uncertainty}, says that
	\[
		\frac{\norm v_1}{\norm{v}_\infty}\cdot \frac{\norm{Av}_1}{\norm{Av}_\infty} \geq k.
	\]
	Following our framework, what we need to prove is the following claim: for every $\delta \in [0,1)$ and for every vector $u$, 
	\[
		 \frac{\ab{\supp_\delta^1(u)}}{1- \delta} \geq \frac{\norm u_1}{\norm u_\infty}
	\]
	Applying this inequality to $v$ with $\varepsilon$ and to $Av$ with $\eta$ yields the desired result, so it suffices to prove this claim. 

	Let $T =\supp_\delta^1(u)$. Note that since $\norm u_1 = \norm{u[T]}_1 + \norm{u[T^{\mathsf{c}}]}_1$, the definition of $T$ implies that $\norm{u[T]}_1 \geq (1- \delta) \norm u_1$.  Observe too that $\norm{u[T]}_\infty = \norm u_\infty$, since $T$ just consists of the coordinates of $u$ of maximal absolute value (with ties broken arbitrarily). Since $u[T]$ has length $\ab T$, this implies that $\norm{u[T]}_1 \leq \ab T \norm{u[T]}_\infty = \ab T \norm u_\infty$. Combining our inequalities, we find that
	\[
		(1- \delta) \norm u_1 \leq \norm{u[T]}_1 \leq \ab T \norm u_\infty = \ab{\supp_\delta^1(u)} \norm u_\infty,
	\]
	as claimed.
\end{proof}

Again, we obtain as a special case an uncertainty principle for the $L^1$ approximate support of a function and its Fourier transform. 
\begin{cor}\label{cor:approx-support}
	Let $G$ be a finite abelian group and $f:G \to \C$ a non-zero function. For any $\varepsilon,\eta \in [0,1]$, we have that
	\[
		\ab{\supp_\varepsilon^1 (f)} \ab{\supp_\eta^1(\hat f)} \geq \ab G (1- \varepsilon)(1- \eta) \geq \ab G (1- \varepsilon- \eta).
	\]
\end{cor}
At first glance, Corollary \ref{cor:approx-support} looks quite similar to Theorem \ref{thm:ds-approx-support}. However, it is easy to construct vectors whose $(2,\varepsilon)$-support is much smaller than their $(1,\varepsilon)$-support. For instance, we can take $v_H \in \C^n$ to be the ``harmonic vector'' $(1,\frac 12, \frac 13,\ldots,\frac 1n)$. Then for any fixed $\varepsilon \in (0,1)$ and large $n$, we have that
\[
	\supp_\varepsilon^1(v_H) = \Theta_\varepsilon(n^{1- \varepsilon}) \qquad \text{ and } \qquad \supp_\varepsilon^2(v_H) = \Theta(\varepsilon^{-2}).
\]
In particular, if we keep $\varepsilon$ fixed and let $n \to \infty$, we see that $\supp_\varepsilon^1(v_H)$ will be much larger than $\supp_\varepsilon^2(v_H)$, which suggests that Theorem \ref{thm:ds-approx-support} will be stronger than Corollary \ref{cor:approx-support} for such ``long-tailed'' vectors. In fact, this is not a coincidence or a special case: for constant $\varepsilon$, the $1$-support will be at least as large than the $2$-support of any vector. More precisely, we have the following.
\begin{prop}\label{prop:1-supp-vs-2-supp}
	For any vector $v \in \C^n$ and any $\varepsilon \in (0,1)$,
	\[
		\ab{\supp_{\varepsilon^2}^1(v)} \geq \ab{\supp_\varepsilon^2(v)}.
	\]
\end{prop}

This proposition demonstrates that in general, Theorem \ref{thm:ds-approx-support} is stronger than our Corollary \ref{cor:approx-support}. Because the proof is somewhat technical, we defer it to Appendix \ref{sec:appendix}.

To conclude this section, we state a generalization of Theorem \ref{thm:ds-approx-support} that applies to all unitary $k$-Hadamard matrices.
\begin{thm}[$L^2$-approximate support uncertainty principle]\label{thm:our-2-support}
	Let $A \in \C^{n \times n}$ be a $k$-Hadamard matrix with $A^*A=kI$, and let $\varepsilon,\eta \in [0,1]$. Let $v \in \C^n$ be a non-zero vector. Then 
	\[
		\ab{\supp_\varepsilon^2(v)} \ab{\supp_\eta^2(Av)} \geq k (1- \varepsilon-\eta)^2.
	\]
\end{thm}
This proof is essentially identical to the original proof from \cite{DoSt}, so we defer it to Appendix \ref{sec:appendix}. We stress that this proof we need to assume the stronger condition that $A^*A=kI$: unlike our other proofs, this proof uses crucially and repeatedly the fact that $\frac{1}{\sqrt k}A$ preserves the $L^2$ norm under this assumption. Similarly, it is important for this proof that the matrix $A$ be square, since we will need the same property for $A^*$. 

We also note that it is impossible to prove this result in the same way we proved Theorem \ref{thm:approx-support}. Indeed, such a proof would necessarily need a $2 \to \infty$ norm uncertainty principle of the form $\norm v_2 \norm{Av}_2 \geq C \norm v_\infty \norm{Av}_\infty$. In the next subsection, we will show (in Theorem \ref{thm:no-norm-up}) that such an inequality cannot hold unless $C$ is  a constant independent of $k$. Therefore, one necessarily has to use an alternative approach (and stronger properties) to prove Theorem \ref{thm:our-2-support}.

\subsection{Uncertainty principles for other norms: possibility and impossibility}

Generalizing the primary uncertainty principle, it is natural and useful to try to prove other uncertainty principles on norms for the finite Fourier transform, or more generally for other $k$-Hadamard operators. Indeed, it seems that one can use Theorem \ref{thm:general-case} directly to derive inequalities of the form $\norm v_p \norm{A v}_p \geq c(k) \norm v_q \norm{A v}_q$ for other norms $p$ and $q$, and for some constant $c(k)$. 

However, the situation for other norms is trickier than for $p=1$ and $q=\infty$. It is instructive to try this for two prototypical cases: first, $p=1$ and $q=2$, and next, $p=2$ and $q=\infty$. In both cases, the constant $c(k)$ we obtain from a  direct use of the general theorem is 1. As we shall see, this happens for different reasons in these two cases. In the first (and generally for $p=1$ and any $q$), one can obtain a much better inequality, indeed a tight one, {\em indirectly} from the case $p=1$ and $q=\infty$.  In the second (and in general for $p\geq 2$ and any $q$), the constant $c(k)=1$ happens to be essentially optimal, and one can only obtain a trivial result. We turn now to formulate and prove each of these statements. 
We note that, besides natural mathematical curiosity, there is a good reason to consider uncertainty principles for other norms: indeed, (a version of) the one we prove here for $p=1$ and $q=2$ will be key to our new proof of Heisenberg's uncertainty principle in Section~\ref{subsec:hup}.

\subsubsection{Optimal norm uncertainty inequalities for \texorpdfstring{$p=1$}{p=1}}
Consider first the case $p=1$ and $q=2$, and suppose we wish to prove such a norm uncertainty principle for the Fourier transform on a finite abelian group $G$. To apply Theorem \ref{thm:general-case}, we would need to scale the Fourier transform matrix into a matrix $A$ with $\norm A_{1 \to 2}, \norm{A^*}_{1 \to 2} \leq 1$. The way to do so is to rescale so all the entries of $A$ have absolute value $1/\sqrt{\ab G}$. But in that case $A^* A =I$, so one would simply obtain the inequality $\norm v_1 \norm{A v}_1 \geq \norm v_2 \norm{A v}_2$, which is not sharp. In fact, this inequality is trivial (and has nothing to do with ``uncertainty''), since any vector $u$ satisfies $\norm u_1 \geq \norm u_2$. 
To obtain a stronger inequality, which is sharp, we instead use a simple  {\em reduction} to the $1 \to \infty$ result of Theorem \ref{thm:hadamard-uncertainty}, which is a variant of the two-step framework articulated in the Introduction.

\begin{thm}[Norm uncertainty principle, $p=1$]\label{thm:hadamard-q-norm}
	Let $A \in \C^{m \times n}$ be $k$-Hadamard, and let $v \in \C^n$ be non-zero. Then for any $1 \leq q \leq \infty$, we have
	\[
		\norm v_1 \norm{Av}_1 \geq k^{1-1/q} \norm v_q \norm{Av}_q.
	\]
\end{thm}
\begin{proof} 
	The case $q=\infty$ is precisely Theorem \ref{thm:hadamard-uncertainty}, so we may assume that $q<\infty$. So, following our approach, all we need to prove is the bound
\begin{equation}\label{eq:q-norm-bound}
     \frac{\norm u_1} {\norm u_q} \geq \left( \frac{\norm u_1} {\norm u_\infty}\right)^{(q-1)/q}
\end{equation}
for any non-zero vector $u$, and plug it in our primary inequality $\norm v_1 \norm{Av}_1 \geq k \norm v_\infty\norm{Av}_\infty$ for both $v$ and $Av$. This is simple: we compute
\begin{align*}
	\norm u_q^q = \sum_{i=1}^n \ab{u_i}^q \leq \norm u_\infty^{q-1} \sum_{i=1}^n \ab{u_i} = \norm u_\infty^{q-1} \norm u_1,
\end{align*}
which implies that $\norm u_1^{q-1} \norm u_q^q \leq \norm u_1^q \norm u_\infty^{q-1}$, which yields the bound (\ref{eq:q-norm-bound}).
\end{proof}
We get as a special case an uncertainty principle for the discrete Fourier transform, which we believe has not been previously observed.
\begin{cor}\label{cor:fourier-q-up}
	For any $1 \leq q \leq \infty$, any finite abelian group $G$, and any non-zero function $f:G \to \C$, we have
	\[
		\norm f_1 \norm{\hat f}_1 \geq \ab G^{1-1/q} \norm f_q \norm{\hat f}_q.
	\]
\end{cor}
\begin{rem}
	This is tight if $f$ is the indicator function of a subgroup $H \subseteq G$. In that case, $\hat f$ is a constant multiple of the indicator function of the dual subgroup $H^\perp \subseteq \wh G$, and we have that $\ab H \ab{H^\perp}=\ab G$. This shows that the result is tight, since the $q$-norm of an indicator function is exactly the $1/q$ power of its support size. 
\end{rem}

\subsubsection{No non-trivial norm uncertainty inequalities for \texorpdfstring{$p \geq 2$}{p≥2}}
Two special cases of Corollary \ref{cor:fourier-q-up}, one of which is just Theorem \ref{thm:hadamard-uncertainty}, are that if $A$ is $k$-Hadamard, then
\[
	\norm v_1 \norm{A v}_1 \geq k \norm v_\infty \norm{A v}_\infty \qquad \text{ and }\qquad \norm v_1 \norm{A v}_1 \geq \sqrt k \norm v_2 \norm{A v}_2.
\]
Looking at these two bounds, it is natural to conjecture that 
\begin{equation}\label{eq:guess}
	\norm v_2 \norm{A v}_2 \geq \sqrt k \norm v_\infty \norm{A v}_\infty,
\end{equation}
which would of course be best possible if true. If we again attempt to prove this for the Fourier transform directly from Theorem \ref{thm:general-case}, we need to scale the Fourier transform to a matrix $A$ with $2 \to \infty$ norm at most $1$, which again requires taking $A$ to have all entries of absolute value $1/\sqrt {\ab G}$. Then we again get that $A^*A=I$, and only obtain the trivial inequality $\norm v_2 \norm{A v}_2 \geq \norm v_\infty \norm{A v}_\infty$. 

In contrast to the previous subsection, this trivial bound is essentially tight, as shown by the following theorem. 
\begin{thm}\label{thm:no-norm-up}
	Let $G$ be a finite abelian group of order $n$, and let $A$ be the Fourier transform matrix of $G$. Let $p \in [2,\infty],q \in [1,\infty]$ be arbitrary. There exists a vector $v \in \C^n$  with
	\[
		\norm v_p \norm{A v}_p \leq 2 \norm v_q \norm{A v}_q.
	\]
	In particular, (\ref{eq:guess}) is false in general.
\end{thm}
\begin{proof} 
	We normalize $A$ so that all its entries have absolute value $1$, and assume without loss of generality that the first row and column of $A$ consist of all ones\footnote{This simply corresponds to indexing the rows and columns of $A$ so that the identity element of $G$ and $\wh G$ come first.}. 
	We define the vector $v=(1+\sqrt n, 1,1,\ldots,1) \in \C^n$. Then $Av=\sqrt n v$, i.e.\ $v$ is an eigenvector of $A$ with eigenvalue $\sqrt n$; this can be seen by observing that $v$ is the sum of $(\sqrt n,0,0,\ldots 0)$ and $(1,1,\ldots,1)$, and the action of $A$ on these vectors is to swap them and multiply each by $\sqrt n$.

	Moreover, we can compute that
	\[
		\norm v_p = \left[ \left(\sqrt n+1\right)^p+n-1 \right] ^{1/p} \qquad \text{and} \qquad \norm v_q =  \left[ \left(\sqrt n+1\right)^q+n-1 \right] ^{1/q}.
	\]
	We claim that for any $a \geq 1,b\geq 0$, the function $h(x) = (a^x+b)^{1/x}$ is monotonically non-increasing for $x\geq 1$. Indeed, its derivative is
	\begin{align*}
		h'(x) &= \frac{h(x)}{x^2(a^x+b)} \left( {a^x \log (a^x)}-(a^x+b)\log(a^x+b) \right).
	\end{align*}
	The term $h(x)/(x^2(a^x+b))$ is positive, and the function $t \mapsto t \log t$ is increasing for $t\geq 1$, which implies that the parenthesized term is non-positive, so $h'(x)\leq 0$. This implies that $\norm v_x$ is a non-increasing function of $x$, so we have that 
	\[
		\norm v_p \leq \norm v_2 = \sqrt{2n+2\sqrt n} \qquad \text{and} \qquad \norm v_q \geq \norm v_\infty = \sqrt n+1.
	\]
	In particular, we find that $\sqrt 2\norm v_q \geq \sqrt 2\norm v_\infty \geq \norm v_2 \geq \norm v_p$. This shows us that 
	\[
		\frac{\norm v_p}{\norm v_q}\cdot \frac{\norm{A v}_p}{\norm{A v}_q} = \frac{\norm v_p^2}{\norm v_q^2}\leq 2. \qedhere
	\]
\end{proof}

\subsubsection{The Hausdorff--Young inequality and the regime \texorpdfstring{$1<p<2$}{1<p<2}}
For the remaining range of $1<p<2$, we can obtain norm inequalities like those for $p=1$. However, we need two additional hypotheses. First, we need to assume that our $k$-Hadamard matrix satisfies the stronger unitarity property that $A^*A = k I$, while such an assumption was unnecessary in the $p=1$ case. Second, we will need to assume that the second norm index, $q$, is at most $p' = p/(p-1)$; this assumption was immaterial in the $p=1$ case, since the dual index of $1$ is $\infty$. We remark that we include this subsection only for completeness; the results here are known and use standard techniques, namely the Riesz--Thorin interpolation theorem and the log-convexity of the $L^p$ norms.

To do this, we first prove a discrete analogue of the Hausdorff--Young inequality. This inequality was already observed\footnote{They used this discrete Hausdorff--Young inequality to prove a discrete entropic uncertainty principle, analogous to that of Hirschman \cite{Hirschman}.} by Dembo, Cover, and Thomas \cite[Equation (52)]{DeCoTh}, who also stated it in the same general setting of unitary $k$-Hadamard matrices, as we now do.
\begin{prop}[Discrete Hausdorff--Young inequality]\label{prop:hausdorff-young}
	Let $A \in \C^{n \times n}$ be a $k$-Hadamard matrix with $A^*A = kI$. Fix $1 < p < 2$, and let $p' =p/(p-1) \in (2,\infty)$. Then $\norm A_{p \to p'} \leq k^{(p-1)/p}$.
\end{prop}
\begin{proof}
	We already know that $\norm A_{1 \to \infty} \leq 1$, and our assumption that $A^* A=kI$ implies that $\norm A_{2 \to 2} = \sqrt k$. We may apply the Riesz--Thorin interpolation theorem \cite[Theorem IX.17]{ReSi} to these bounds, which implies that $\norm A_{p \to p'} \leq k^{(p-1)/p}$, as claimed.
\end{proof}
As a corollary, we obtain the following norm uncertainty principle for $1<p<2$.
\begin{thm}[Norm uncertainty principle, $1<p<2$]\label{thm:1<p<2}
	Let $A \in \C^{n \times n}$ be a $k$-Hadamard matrix with $A^*A = kI$. Let $p \in (1,2)$ and $q \in [p,p']$ be norm indices. Then for any $v \in \C^n$,
	\[
		\norm v_p \norm{Av}_p \geq k^{\frac{q-p}{pq}} \norm v_q \norm{Av}_q
	\]
\end{thm}
\begin{proof}
	First, suppose that $q=p'$. In that case, we may multiply the conclusion of Proposition \ref{prop:hausdorff-young} for $A$ and $A^*$ and conclude that
	\[
		k^{\frac{2(p-1)}p} \norm v_p \norm{Av}_p \geq \norm {Av}_{p'} \norm{A^* Av}_{p'} =  k \norm v_{p'}\norm{Av}_{p'},
	\]
	which implies the desired bound $\norm v_p \norm{Av}_p \geq k^{\frac{2-p}{p}} \norm v_{p'} \norm{Av}_{p'}$. 

	For smaller values of $q$, we use the above as a primary uncertainty principle, and derive the result by showing that
	\begin{equation}\label{eq:p-q-p'}
		\frac{\norm u_p}{\norm u_q} \geq \left(\frac{\norm u_p}{\norm u_{p'}}\right)^{\frac{p-q}{pq-2q}}
	\end{equation}
	for any non-zero vector $u$ and any $p \leq q \leq p'$. Let $\theta \in [0,1]$ be the unique number such that
	\[
		\frac 1q = \frac{1-\theta}{p}+ \frac{\theta}{p'},
	\]
	namely $\theta = \frac{p-q}{pq - 2q}$. 
	Then the log-convexity of the $L^p$ norms (also known as the generalized H\"older inequality) says that $\norm u_q \leq \norm u_p^{1- \theta} \norm u_{p'}^\theta$, and rearranging this yields (\ref{eq:p-q-p'}).

	We now apply (\ref{eq:p-q-p'}) to $u=v$ and $u=Av$, and conclude that
	\[
		\frac{\norm v_p}{\norm v_{q}} \cdot \frac{\norm{Av}_p}{\norm{Av}_{q}} \geq \left( \frac{\norm v_p}{\norm v_{p'}} \cdot \frac{\norm{Av}_p}{\norm{Av}_{p'}} \right) ^{\frac{p-q}{pq-2q}} \geq \left( k^{\frac{2-p}{p}} \right) ^{\frac{p-q}{pq-2q}} = k^{\frac{q-p}{pq}}.\qedhere
	\]
\end{proof}
\begin{rem}
	We remark that, as with the $1 \to q$ norm uncertainty principles above, Theorem \ref{thm:1<p<2} is tight for the Fourier transform. Indeed, if $v$ is the indicator vector of a subgroup of $G$, then $Av$ will be a constant multiple of the indicator vector of the dual subgroup, and the inequality in Theorem \ref{thm:1<p<2} will be an equality in this case.
\end{rem}

What these subsections demonstrate is that the $1 \to \infty$ result of Theorem \ref{thm:hadamard-uncertainty} is the strongest result of its form, in two senses. First, it implies the optimal $1 \to q$ inequalities for any $1 \leq q \leq \infty$, which cannot be obtained by a direct application of Theorem \ref{thm:general-case}. Second, such $p \to q$ uncertainty principles for $p>1$ are false in general, as shown by the fact that one cannot obtain a super-constant uncertainty even for the Fourier transform when $p \geq 2$. In the regime $1<p<2$, one can obtain tight inequalities whenever $p \leq q \leq p'$, at least for $k$-Hadamard matrices that satisfy the unitarity property $A^*A=kI$.

\section{Uncertainty principles in infinite dimensions}\label{sec:infinite-dim}
In this section, we will state and prove various uncertainty principles that hold in infinite-dimensional vector spaces, primarily the Heisenberg uncertainty principle and its generalizations. We begin in Section \ref{subsec:arbitrary-groups} with general results that hold for the Fourier transform on arbitrary locally compact abelian groups. We then restrict to $\R$, and discuss in Section \ref{subsec:k-Hadamard-infinite} a large class of operators for which our results hold, namely infinite-dimensional analogues of the $k$-Hadamard matrices we focused on in Section \ref{sec:finite-dim}. These include the so-called Linear Canonical Transforms (LCT), a family of integral transforms generalizing of the Fourier and other transforms, which arise primarily in applications to optics. Finally, we move to prove the Heisenberg uncertainty principle and its variations for such operators in Section \ref{subsec:hup}. In addition to obtaining a new proof which avoids using the analytic tools common in existing proofs, we also prove a number of generalizations. Most notably, we establish uncertainty principles for higher moments than the variance\footnote{Such results were already obtained by Cowling and Price \cite{CoPr}, but again, our proof avoids their heavy analytic machinery.}. We also give new inequalities which are similar to Heisenberg's but are provably incomparable. We remark that in some of our proofs of existing inequalities, the constants obtained are worse than in the classical proofs.

\subsection{The Fourier transform on locally compact abelian groups}\label{subsec:arbitrary-groups}
We begin by recalling the basic definitions of the Fourier transform on locally compact abelian\footnote{In fact, we believe that, as in Section \ref{subsec:non-abelian-gps}, many of our results can be extended to infinite non-abelian groups, at least as long as all their irreducible representations are finite-dimensional.} groups, and proving some generalizations of earlier results in this context. 

Let $G$ be a locally compact abelian group. Then $G$ can be equipped with a left-invariant Borel measure $\mu$, called the \emph{Haar measure}, which is unique up to scaling. If we let $\wh G$ denote the set of continuous group homomorphisms $G \to \T$, then $\wh G$ is a group under pointwise multiplication. Moreover, if we topologize $\wh G$ with the compact-open topology, then $\wh G$ becomes another locally compact abelian group, which is called the \emph{Pontryagin dual} of $G$. Given a function $f \in L^1(G)$, we can define its Fourier transform $\hat f:\wh G \to \C$ by $\hat f(\chi) = \int f(x) \ol{\chi(x)} \dd \mu(x)$, and it is easy to see that $\hat f$ is a well-defined element of $L^\infty(\wh G)$. Moreover, having chosen $\mu$, there exists a unique Haar measure $\nu$ on $\wh G$ so that the \emph{Fourier inversion formula} holds, namely so that $f(x) = \int \hat f(\chi) \chi(x) \dd \nu(\chi)$ for $\mu$-a.e.\ $x$, as long as $\hat f \in L^1(\wh G)$. With this choice of $\nu$, we also have the \emph{Plancherel formula}, that $\int \ab{f}^2 \dd \mu = \int \ab{\hat f}^2 \dd \nu$, as long as one side is well-defined. From now on, we will fix these measures $\mu$ and $\nu$, and all $L^p$ norms of functions will be defined by integration against these measures. Observe that from the definition of $\hat f$ and from the Fourier inversion formula, we have that the Fourier transform and inverse Fourier transform have norm at most $1$ as operators $L^1 \to L^\infty$. Using this, we can prove an infinitary version of our primary uncertainty principle, Theorem \ref{thm:primary-up}. 

\begin{thm}[Primary uncertainty principle, infinitary version]\label{thm:infinite-primary-up}
	Let $G$ be a locally compact abelian group with a Haar measure $\mu$, and let $\wh G,\nu$ be the dual group and measure. Fix $1 \leq q \leq \infty$ and let $f \in L^1(G)$ be such that $\hat f \in L^1(\wh G)$. Then 
	\[
		\norm f_1 \norm{\hat f}_1 \geq \norm{f}_\infty \norm{\hat f}_\infty.
	\]
\end{thm}
\begin{rem}
	Throughout this section, we will frequently need the assumption that $f \in L^1(G)$ and $\hat f \in L^1(\wh G)$. To avoid having to write this every time, we make the following definition.
	\begin{Def}[Doubly $L^1$ function]
		We call function $f:G \to \C$ \emph{doubly $L^1$} if $f \in L^1(G)$ and $\hat f \in L^1(\wh G)$. 
	\end{Def}
	Note that $f$ being doubly $L^1$ implies that $f, \hat f \in L^\infty$, and thus that $f,\hat f \in L^p$ for all $p \in [1,\infty]$ by H\"older's inequality.
\end{rem}
\begin{proof}[Proof of Theorem \ref{thm:infinite-primary-up}]
	For any $\chi \in \wh G$, we have that
	\[
		\ab{\hat f(\chi)} = \bab{\int f(x) \ol{\chi(x)} \dd \mu(x)}\leq \int \ab{f(x)} \dd \mu(x) = \norm f_1,
	\]
	since $\ab{\chi(x)}=1$. This implies that $\norm {\hat f}_\infty \leq \norm f_1$. For the same reason, we see that $\norm f_\infty \leq \norm{\hat f}_1$. Multiplying these inequalities gives the desired result. 
\end{proof}
\begin{rem}
	If we take $G$ to be a finite abelian group, this result appears to be a factor of $\ab G$ worse than Theorem \ref{thm:hadamard-uncertainty}. However, this discrepancy is due to the fact that previously, we were equipping both $G$ and $\wh G$ with the counting measure, which are not dual Haar measures. If we instead equip them with dual Haar measures (e.g.\ equipping $G$ with the counting measure and then equipping $\wh G$ with the uniform probability measure), then this ``extra'' factor of $\ab G$ would disappear, and we would get the statement of Theorem \ref{thm:infinite-primary-up}.
\end{rem}

Using this theorem, we can obtain an analogue of the Donoho--Stark uncertainty principle, which holds for every locally compact abelian group. This result was first proved by Matolcsi and Sz\H ucs \cite{MaSz}, using the theory of spectral integrals.
\begin{thm}[Support-size uncertainty principle for general abelian groups]\label{thm:infinite-ds}
	Let $G,\mu,\wh G, \nu$ be as above. Let $f:G \to \C$ be non-zero and doubly $L^1$. Then $\mu(\supp(f)) \nu(\supp(\hat f)) \geq 1$. 
\end{thm}
\begin{proof}
	The proof follows that of Theorem \ref{thm:hadamard-ds}. Following our general approach, we claim that for any non-zero integrable function $g$ on any measure space $(X,\lambda)$, we have that 
	\[
		\lambda(\supp(g)) \geq \frac{\norm g_1}{\norm g_\infty}.
	\]
	Applying this to $f$ and $\hat f$ and combining it with the primary uncertainty principle, Theorem \ref{thm:infinite-primary-up}, yields the desired result. To prove the claim, we simply compute
	\[
		\norm g_1 = \int_X \ab{g(x)} \dd \lambda(x) = \int_{\supp(g)} \ab{g(x)} \dd \lambda(x) \leq \norm g_\infty \int_{\supp(g)} \dd \lambda(x) = \lambda(\supp(g)) \norm g_\infty. \qedhere
	\]
\end{proof}
In general, Theorem \ref{thm:infinite-ds} is tight. This can be seen, for instance, by recalling that it is equivalent to Theorem \ref{thm:donoho-stark} when $G$ is finite, and we already know that theorem to be tight when $f$ is the indicator function of a subgroup. However, Theorem \ref{thm:infinite-ds} is tight even for some infinite groups. For instance, let $G$ be any compact abelian group, and let $\mu$ be the Haar probability measure on $G$. Then $\wh G$ is a discrete group, and $\nu$ is the counting measure on $\wh G$. If we let $f:G \to \C$ be the constant $1$ function, then $\mu(\supp(f))=1$. Moreover, $\hat f$ will be the indicator function of the identity in $\wh G$, so $\nu(\supp(\hat f))=1$ as well. 

However, when we restrict to $G=\R$ and $\mu$ the Lebesgue measure, we find that Theorem \ref{thm:infinite-ds} is far from tight. Instead, the correct inequality is $\mu(\supp(f)) \nu(\supp(\hat f))=\infty$, as proven by Benedicks \cite{Benedicks} and strengthened by Amrein and Berthier \cite{AmBe}. The proofs of these results use the specific structure of $\R$, and we are not able to reprove them with our framework, presumably because our approach should work for any $G$, and Benedicks's result is simply false in general. There has been a long line of work on how much Theorem \ref{thm:infinite-ds} can be strengthened for other locally compact abelian groups $G$; see \cite[Section 7]{FoSi} for more.

\subsection{\texorpdfstring{$k$}{k}-Hadamard operators in infinite dimensions}\label{subsec:k-Hadamard-infinite}
Continuing to restrict to functions on $\R$, one can ask for other transforms which satisfy an uncertainty principle, just as previously we investigated all $k$-Hadamard matrices, and not just the Fourier transform matrices. From the proof of Theorem \ref{thm:infinite-primary-up}, and from the definition of $k$-Hadamard matrices, the following definition is natural. 
\begin{Def}
	We say that a linear operator $A:L^1(\R) \to L^\infty(\R)$ is \emph{$k$-Hadamard} if $\norm A_{1 \to \infty} \leq 1$ and if $\norm{A^*Af}_\infty \geq k \norm f_\infty$ for all functions $f$ with $f,Af \in L^1(\R)$. 
\end{Def}
\begin{rem}
	Extending our earlier use of the word, we will say that $f$ is \emph{doubly $L^1$ for $A$} if $f,Af \in L^1(\R)$. We will usually just say \emph{doubly $L^1$} and omit ``for $A$'' when $A$ is clear from context.
\end{rem}
The primary uncertainty principle for $k$-Hadamard operators, extending Theorem \ref{thm:infinite-primary-up}, is the following, whose proof is identical to that of Theorem \ref{thm:infinite-primary-up}.
\begin{thm}[Primary uncertainty principle for $k$-Hadamard operators]\label{thm:k-Hadamard-primary}
	Suppose $A$ is a $k$-Hadamard operator and $f$ is doubly $L^1$. Then
	\[
		\norm f_1 \norm{Af}_1 \geq k \norm f_\infty \norm{Af}_\infty.
	\]
\end{thm}
We can also extend the uncertainty principles for other norms seen in Theorem \ref{thm:hadamard-q-norm} to this infinite-dimensional setting, as follows.
\begin{thm}[Norm uncertainty principle, infinitary version]\label{thm:expanding-q-norm}
	Suppose $A$ is a $k$-Hadamard operator and $f$ is doubly $L^1$. Then for any $1 \leq q \leq \infty$,
	\[
		\norm f_1 \norm{Af}_1\geq k^{1-1/q} \norm f_q \norm{Af}_q.
	\]
\end{thm}
\begin{proof}
	The proof follows that of Theorem \ref{thm:hadamard-q-norm}. We may assume that $q<\infty$, since the case of $q=\infty$ is precisely Theorem \ref{thm:k-Hadamard-primary}. It suffices to prove that for any non-zero function $g \in L^1(\R) \cap L^\infty(\R)$,
	\begin{equation}\label{eq:q-norm-bound-function}
		\frac{\norm g_1}{\norm g_q} \geq \left( \frac{\norm g_1}{\norm g_\infty} \right) ^{(q-1)/q},
	\end{equation}
	since we may then apply this bound to $f$ and $\hat f$ and use the primary uncertainty principle, Theorem \ref{thm:k-Hadamard-primary}. To prove (\ref{eq:q-norm-bound-function}), we simply compute
	\[
		\norm g_q^q = \int \ab{g(x)}^q \dd x \leq \norm g_\infty^{q-1} \int \ab{g(x)} \dd x = \norm g_\infty^{q-1} \norm g_1,
	\]
	which implies (\ref{eq:q-norm-bound-function}) after multiplying both sides by $\norm g_1^{q-1}$ and rearranging.
\end{proof}
We already saw in the previous section that the Fourier transform on $\R$ is $1$-Hadamard. As it turns out, the Fourier transform is one instance of a large class of $k$-Hadamard operators (with arbitrary values of $k$) known as \emph{linear canonical transformations (LCT)}, which we define below. These transformations arise in the study of optics, and generalize many other integral transforms on $\R$, such as the fractional Fourier and Gauss--Weierstrass transformations. Although their analytic properties are somewhat more complicated than those of the Fourier transform, our framework treats them equally, since the only property we will need of them is that they are $k$-Hadamard. For more information on LCT, see \cite[Chapters 9--10]{Wolf} or \cite{LCT}. 

We now define the LCT, following \cite{BaAl}. This is a family of integral transforms, indexed by the elements of $\slr$. Specifically, given a matrix $M=\smat{a&b\\c&d} \in \slr$ with $b \neq 0$, we can define the LCT $L_M$ associated to $M$ to be 
\[
	(L_M f)(\xi) = \frac{e^{-i\pi \sgn(b)/4}}{\sqrt{\ab b}} \int f(x) e^{i \pi (d \xi^2 - 2 x \xi +a x^2)/b} \dd x.
\]
One can also take the limit $b \to 0$ and obtain a consistent definition of $L_M$ for all $M \in \slr$. It turns out that this definition yields an infinite-dimensional representation of $\slr$; in particular, one sees that the inverse transform $L_M\inv$ is given by $L_{M\inv}=(L_M)^*$. From the definition, we see that if $b \neq 0$,
\[
	\ab{(L_Mf)(\xi)} = \frac{1}{\sqrt{\ab b}} \bab{\int f(x) e^{i \pi (d \xi^2 - 2 x \xi +a x^2)/b} \dd x} \leq \frac{1}{\sqrt{\ab b}} \int \ab{f(x)} \dd x = \frac{\norm f_1}{\sqrt{\ab b}},
\]
so $\norm{L_M}_{1 \to \infty} \leq 1/\sqrt{\ab b}$. This implies the following result.
\begin{thm}
	Let $M=\smat{a&b\\c&d} \in \slr$ be a matrix with $b \neq 0$. Let $A=\sqrt{\ab b}L_M$ be a rescaling of the LCT $L_M$. Then $A$ is $\ab b$-Hadamard.
\end{thm}
\begin{proof}
	By the above, we see that $\norm{A}_{1 \to \infty} =\sqrt{\ab b} \norm{L_M}_{1 \to \infty} \leq 1$. Similarly, if we set $B=A^*=\sqrt{\ab b} L_{M\inv}$, then $\norm B_{1 \to \infty} \leq 1$ and $BAf = \ab b L_{M\inv} L_Mf=\ab b f$ for any doubly $L^1$ function $f$.
\end{proof}
By combining the primary uncertainty principle for $k$-Hadamard operators with the argument of Theorem \ref{thm:infinite-ds}, we obtain the following generalization of the Matolcsi--Sz\H ucs (or Donoho--Stark) uncertainty principle for the LCT, or indeed for any $k$-Hadamard operator.
\begin{cor}
	If $M=\smat{a&b\\c&d} \in \slr$ and $f:\R \to \C$ is doubly $L^1$ and non-zero, then
	\[
		\lambda(\supp(f)) \lambda(\supp(L_M f)) \geq \ab b,
	\]
	where $\lambda$ denotes Lebesgue measure. 
\end{cor}
\begin{proof}
	From the primary uncertainty principle, Theorem \ref{thm:k-Hadamard-primary}, we find that
	\[
		\frac{\norm f_1}{\norm f_\infty} \cdot \frac{\norm{L_M f}_1}{\norm{L_M f}_\infty} \geq \ab b.
	\]
	In proving Theorem \ref{thm:infinite-ds}, we saw that $\frac{\norm g_1}{\norm g_\infty} \leq \lambda(\supp(g))$ for all $g$, which yields the claim.
\end{proof}
We believe that this fact was not previously observed for the LCT. Of course, one expects that in general a much stronger result should hold, namely that $\lambda(\supp(f)) \lambda(\supp(L_M f)) =\infty$ whenever $b \neq 0$; this would generalize the result of Benedicks \cite{Benedicks} from the Fourier transform to all LCT. However, we are not able to obtain such a result with our approach, for the same reason that we cannot reprove Benedicks's theorem.

\subsection{The Heisenberg uncertainty principle}\label{subsec:hup}
In this section, we prove (with a somewhat worse constant) the well-known Heisenberg uncertainty principle, as well as some extensions of it. Again, as in all previous proofs we have seen, we use the elementary two-step process explained in the Introduction. 
Our proof differs drastically from the classical ones, which use analytic techniques (integration by parts) and special properties of the Fourier transform (that it turns differentiation into multiplication by $x$). Indeed it is not clear if these classical techniques can be used to prove our generalizations.

For a doubly $L^1$ function $f$, we define the \emph{variance} of $f$ to be
\[
	V(f) = \int x^2 \ab{f(x)}^2 \dd x.
\]
If $\norm f_2 = 1$, then we may think of $\ab f^2$ as a probability distribution, in which case $V$ really does measure the variance of this distribution (assuming, without loss of generality\footnote{If its mean is at some point $a$, we can simply replace $f(x)$ by $f(x-a)$ to make $V$ be the actual variance of the distribution.}, that its mean is $0$). This interpretation is natural from the perspective of quantum mechanics (whence the original motivation for studying uncertainty principles): in quantum mechanics, we would think of $f$ as a wave function, and then $\ab f^2$ would define the probability distribution for measuring some quantity associated to the wave function, such as a particle's\footnote{Because of this interpretation, it is natural to have $f$ be a function defined on $\R^n$, to model a particle moving in $n$-dimensional space. For the moment we focus on the case $n=1$, though we discuss the multidimensional analogue in Section \ref{sec:open-problems}.} position or momentum. Heisenberg's\footnote{Though the physical justification for the uncertainty principle is due to Heisenberg \cite{Heisenberg}, the proof of the mathematical fact is due to Kennard \cite{Kennard} and Weyl \cite{Weyl}.} 
uncertainty principle asserts that $V(f)$ and $V(\hat f)$ cannot both be small.
\begin{thm}[Heisenberg's uncertainty principle \cite{Heisenberg,Kennard,Weyl}]\label{thm:original-heisenberg}
	There exists a constant $C>0$ 
	such that for any doubly $L^1$ function $f \neq 0$,
	\[
		V(f) V(\hat f) \geq C \norm f_2^2 \norm{\hat f}_2^2.
	\]
\end{thm}
\begin{rem}
	It is in fact known that the optimal constant is $C=1/(16 \pi ^2)$, with equality attained for Gaussians.
\end{rem}
Additionally, versions of the Heisenberg uncertainty principle have been established for the LCT, see \cite{TaZh} for a survey. The most basic such extension is the following, stated without proof as \cite[Exercise~9.10]{Wolf} and first proven in print by Stern \cite{Stern}.
\begin{thm}[LCT uncertainty principle \cite{Stern,Wolf}]
	There exists a constant $C>0$ such that the following holds for all doubly $L^1$ functions $f$. If $M = \smat{a&b\\c&d} \in \slr$ and $L_M$ is the associated LCT, then
	\[
		V(f) V(L_M f) \geq C b^2 \norm f_2^2 \norm{L_M f}_2^2.
	\]
\end{thm}

\subsubsection{A Heisenberg uncertainty principle for other norms}
We begin by proving the following generalization of Heisenberg's uncertainty principle.  It lets us bound $V(f) V(A f)$ by any $L^q$ norm of $f$ and $A f$, for any $k$-Hadamard operator $A$ (recovering, for $q=2$, the classical results of the previous subsection\footnote{Note e.g.\ that one recovers the correct dependence on $b$ when deducing the LCT uncertainty principle above from this one.}). As far as we know, this result is new for $q \neq 2$, even for the Fourier transform. As we show below, the statements for different $q$ are in general of incomparable strength.
\begin{thm}[Heisenberg uncertainty principle for arbitrary norms]\label{thm:our-heisenberg}
	For any $k$-Hadamard operator $A$, any doubly $L^1$ function $f$, and any $1 < q \leq \infty$,
	\[
		V(f) V(A f) \geq C_q k^{3-2/q} \norm f_q^2 \norm {\hat f}_q^2,
	\]
	where $C_q=2^{-\frac{10q-8}{q-1}}$ depends only on $q$. In particular, $V(f) V(\hat f) \geq C_q \norm f_q^2 \norm{\hat f}_q^2$.
\end{thm}
\begin{rem}
	No attempt was made to optimize the constant $C_q$. However, our proof is unlikely to give the optimal constant even after optimization; for instance, in the case $q=2$, it is known that the optimal constant for the Fourier transform is $C_2=1/(16\pi^2) \approx 6.3 \times 10^{-3}$, whereas our proof gives the somewhat worse constant $2^{-12} \approx 2.4 \times 10^{-4}$.
\end{rem}
As with our other proofs, the proof of this result proceeds in two stages. The first, already done in Theorem \ref{thm:expanding-q-norm},
is establishing the ``norm uncertainty principle'' $\norm f_1 \norm{A f}_1 \geq \norm f_q \norm{A f}_q$. After this, all that remains is to lower-bound $V(g)$ as a function of $\norm g_1$ and $\norm g_q$ for an arbitrary function $g$. Combining these two bounds will yield the result.

However, an important new ingredient which we did not use in the finite-dimensional setting is a different way to upper-bound $\norm g_1$. The idea is to choose a constant $T$, depending on $g$ and the target norm $q$, so that most of the the $L^1$-mass of $g$ is outside the interval $[-T,T]$. 
This will allow us to lower bound the variance through a simple use of H\"older's inequality.  Note that in the proof and subsequently, we use the usual conventions of manipulating $q$ as though it is finite, though everything works identically for $q=\infty$ by taking a limit, or by treating expressions like $\infty/(\infty-1)$ as equal to $1$.

\begin{proof}[Proof of Theorem \ref{thm:our-heisenberg}]
	We may assume that $f \neq 0$. Following our general framework, we claim that the bound
	\begin{equation}\label{eq:variance-bound}
		\frac{\norm g_1}{\norm g_q} \leq  \left( 2^{\frac{5q-4}{q-1}} \frac{V(g)}{\norm g_q^2} \right) ^{\frac{q-1}{3q-2}}
	\end{equation}
	holds for any non-zero function $g \in L^1(\R) \cap L^\infty(\R)$. Observe that this bound is homogeneous, in that it is unchanged if we replace $g$ by $cg$ for some constant $c$. Once we have this bound, we can apply it to the norm uncertainty principle, Theorem \ref{thm:expanding-q-norm}, which says that
	\[
		\frac{\norm f_1}{\norm f_q} \cdot \frac{\norm {Af}_1}{\norm{Af}_q} \geq k^{1-1/q}.
	\]
	Plugging in (\ref{eq:variance-bound}) for $g=f$ and $g=Af$, we find that
	\[
		\left( 2^{\frac{10q-8}{q-1}} \frac{V(f)}{\norm f_q^2}\frac{V(Af)}{\norm {Af}_q^2} \right) ^{\frac{q-1}{3q-2}} \geq k^{\frac{q-1}{q}},
	\]
	and rearranging gives the desired conclusion. So it suffices to prove (\ref{eq:variance-bound}).

	Let $T=\frac 12 (\norm g_1 /(2 \norm g_q))^{q/(q-1)}$, so that $(2T)^{1-1/q} \norm g_q = \frac 12 \norm g_1$. By H\"older's inequality, we have that
	\[
		\int_{-T}^T \ab{g(x)} \dd x = \int \bm 1_{[-T,T]}(x) \ab{g(x)} \dd x \leq \norm{\bm 1_{[-T,T]}}_{q/(q-1)} \norm g_q = (2T)^{1-1/q} \norm g_q = \frac 12 \norm g_1,
	\]
	where the last step follows from the definition of $T$. This implies that the interval $[-T,T]$ contains at most half of the $L^1$ mass of $g$, so $\frac 12 \norm g_1 \leq \int_{\ab x>T} \ab{g(x)} \dd x$. Applying the Cauchy--Schwarz inequality to this bound, we find that
	\begin{align*}
		\frac 12 \norm g_1 & \leq \int_{\ab x >T} \ab{g(x)} \dd x\\
		&= \int_{\ab x>T} \frac 1x (x \ab{g(x)}) \dd x\\
		&\leq \left( \int_{\ab x>T} \frac{1}{x^2} \dd x \right) ^{1/2} \left( \int_{\ab x>T} x^2 \ab{g(x)}^2 \dd x \right) ^{1/2}\\
		&\leq \left( \frac{2}{T} \right) ^{1/2} V(g)^{1/2}\\
		&= 2\left( \frac{2 \norm g_q}{\norm g_1} \right) ^{q/(2(q-1))} V(g)^{1/2}.
	\end{align*}
	Rearranging this inequality yields (\ref{eq:variance-bound}).
\end{proof}
Theorem \ref{thm:our-heisenberg} contains within it infinitely many ``Heisenberg-like'' uncertainty principles, one for each $q \in (1,\infty]$ (and one for each $k$-Hadamard operator $A$). It is natural to wonder whether these are really all different, or whether one of them implies all the other ones. As a first step towards answering this question in the case of the Fourier transform, we can show that the $q=2$ and $q=\infty$ cases are incomparable, in the sense that there exist functions for which one is arbitrarily stronger than the other. More precisely, we have the following result. We use $\sch(\R)$ to denote the Schwartz class of rapidly decaying smooth functions.
\begin{thm}\label{thm:incomparable-heisenbergs}
	Define a function $F:\sch(\R) \setminus \{0\} \to \R_{>0}$ by
	\[
		F(f) = \frac{\norm f_\infty \norm{\hat f}_\infty}{\norm f_2 \norm{\hat f}_2}= \frac{\norm f_\infty \norm{\hat f}_\infty}{\norm f_2^2}.
	\]
	Then the image of $F$ is all of $\R_{>0}$.
\end{thm}
\noindent We defer the proof of Theorem \ref{thm:incomparable-heisenbergs} to Appendix \ref{sec:appendix}. 

Recall that the $q=2$ case of Theorem \ref{thm:our-heisenberg} (which is just the classical Heisenberg uncertainty principle) says that $V(f) V(\hat f) \geq C \norm f_2^2 \norm{\hat f}_2^2$, whereas the $q=\infty$ case of Theorem \ref{thm:our-heisenberg} says that $V(f) V(\hat f) \geq C' \norm f_\infty^2 \norm{\hat f}_\infty^2$, for appropriate constants $C,C'>0$. Thus, Theorem \ref{thm:incomparable-heisenbergs} says that these two results are in general incomparable: there exist functions for which the $q=2$ gives an arbitrarily stronger lower bound on $V(f) V(\hat f)$, while there exist other functions for which the $q=\infty$ case gives an arbitrarily stronger bound. In fact, we expect that in general, the uncertainty principles for any $q\neq 2$ should be incomparable to the Heisenberg uncertainty principle, namely the case where $q=2$. We are unable to prove this, and therefore leave it as a conjecture. 
\begin{conj}
	Let $2 \neq q \in (1,\infty]$, and define a function $F_{q}:\sch(\R) \setminus \{0\} \to \R_{>0}$ by
	\[
		F_q(f)=\frac{\norm f_{q} \norm{\hat f}_{q}}{\norm f_2 \norm{\hat f}_2}=\frac{\norm f_{q} \norm{\hat f}_{q}}{\norm f_2^2}.
	\]
	Then the image of $F_q$ is all of $\R_{>0}$. 
\end{conj}

\subsubsection{An uncertainty principle for higher moments}
Theorem \ref{thm:our-heisenberg} is itself a special case of a much more general uncertainty principle, which we now state. Rather than proving uncertainty for the variance functional $V(f)$, it proves it for any moments greater than $1$ 
of the distributions $\ab{f}^2$ and $\ab{A f}^2$. Namely, for any $1<r<\infty$, let us define
\[
	M_r(f) = \int \ab x^r \ab{f(x)}^2 \dd x,
\]
which is precisely the $r$th moment of the distribution $\ab f^2$ if $\norm f_2=1$. Even when $\norm f_2 \neq 1$, $M_r(f)$ is still a good measure of how ``spread'' $f$ is, in that it computes how much $L^2$ mass of $f$ is far from the origin, weighted according to $\ab x^r$. 
Uncertainty principles for such functionals were studied by Cowling and Price \cite{CoPr}, who established the $q=2$ case of the following result for the Fourier transform, as well as many more general results of this flavor for the Fourier transform. As with the basic Heisenberg uncertainty principle, we believe that the $q \neq 2$ case is new, as is the extension to arbitrary $k$-Hadamard operators.
\begin{thm}[Heisenberg uncertainty principle for higher moments]\label{thm:general-heisenberg}
	Let $1<r,s<\infty$ and $1<q \leq \infty$. Then for any $k$-Hadamard operator $A$ and any doubly $L^1$ function $f$,
	\[
		M_r(f)^{\frac{q-1}{qr+q-2}} M_s(A f)^{\frac{q-1}{qs+q-2}} \geq C_{r,q} C_{s,q} k^{1-\frac 1q}\norm f_q^{\frac{2q-2}{qr+q-2}} \norm{A f}_q^{\frac{2q-2}{qs+q-2}},
	\]
	for some constants $C_{r,q}, C_{s,q}>0$ depending only on $r,q$, and $s,q$, respectively. In particular, if $s=r$, we have
	\[
		M_r(f) M_r(A f) \geq C_{r,q}' k^{\frac{qr+q-2}{q}} \norm f_q^2 \norm{\hat f}_q^2
	\]
	for another constant $C_{r,q}' = C_{r,q}^{2(qr+q-2)/(q-1)}$. 
\end{thm}
We defer the proof of this theorem to Appendix \ref{sec:appendix}, but the basic idea is the same as the proof of Theorem \ref{thm:our-heisenberg}: by the general framework, it suffices to upper-bound $\norm f_1/\norm f_q$ as a function of $M_r(f)$. To do so, we again choose an appropriate $T$ so that most of the $L^1$ mass of $f$ is outside of $[-T,T]$, and then proceed by simple applications of the H\"older and Cauchy--Schwarz inequalities. 

\subsubsection{Further extensions and open questions}\label{sec:open-problems}
Cowling and Price \cite{CoPr} actually study a much more general question than what is in Theorem \ref{thm:general-heisenberg} (though restricting to $q=2$ and the Fourier transform). They investigate integrals of the form $\int w(x) \ab{f(x)}^p$ for values of $p$ other than $2$ and for quite general functions $w$, and finding necessary and sufficient conditions for an uncertainty principle to hold for such functionals of $f$ and $\hat f$. Using the same techniques, we can also obtain such results in the case $w(x)=\ab x^r$, as above; the proof is identical to that of Theorem \ref{thm:general-heisenberg}, except that instead of using Cauchy--Schwarz to write
\[
		\int_{\ab x>T} \ab{f(x)} \dd x = \int_{\ab x >T} \frac{1}{\ab x^{r/2}} (\ab x^{r/2} \ab{f(x)}) \dd x \leq \left( \int_{\ab x>T} \frac{\mathrm d x}{\ab x^r} \right) ^{1/2} \left( \int \ab x^r \ab{f(x)}^2 \dd x \right) ^{1/2},
\]
we would instead use H\"older's inequality to say
\[
	\int_{\ab x>T} \ab{f(x)} \dd x = \int_{\ab x>T} \frac{1}{\ab x^{r/p}} (\ab x^{r/p} \ab{f(x)}) \dd x \leq \left( \int_{\ab x>T} \frac{\mathrm d x}{\ab x^{\frac{r}{p-1}}} \right)^{\frac{p-1}{p}}  \left( \int \ab x^r \ab{f(x)}^p \right) ^{\frac 1p}.
\]
Then as long as $r>p-1$, this first integral will converge, and the argument would go through as before. We omit the details, since they are very similar to (but messier than) the computations in the proof of Theorem \ref{thm:general-heisenberg}.

However, an interesting point is raised by this argument, which is the fact that it only works for $r>p-1$. Cowling and Price's theorem works for all $r>(p-2)/2$, which is a larger range, and they in fact prove a converse which says that no such result is true if $r$ is any smaller. It is not at present clear to us whether there is a genuine obstruction that prevents our technique from working for all possible $r$, or if there is some variant manipulation that would yield the full strength of Cowling and Price's theorem.

A similar convergence issue arises when attempting to prove the multidimensional version of Heisenberg's uncertainty principle with our framework. The multidimensional version says the following.
\begin{thm}\label{thm:multidim-hup}
Let $n \geq 1$ be an integer and $f \in L^2(\R^n)$. Then
\begin{equation}
	\left(\int \norm x_2^2 \ab{f(x)}^2 \dd x\right)\left(\int \norm \xi_2^2 \ab{\hat f(\xi)}^2 \dd \xi\right) \geq C n^2 \norm f_2^2 \norm{\hat f}_2^2,
\end{equation}
where the constant $C>0$ does not depend on the dimension $n$. 
\end{thm}
If one attempts to prove this by using the argument from the proof of Theorem \ref{thm:our-heisenberg}, the natural thing to try is to pick $T$ appropriately and then to write
\[
	\frac 12 \norm f_1 \leq \int_{\norm x_2 >T} \ab{f(x)}\dd x \leq \left( \int_{\norm x_2>T} \frac{1}{\norm x_2^2} \dd x \right) ^{1/2} \left( \int_{\norm x_2 >T} \norm x_2^2 \ab{f(x)}^2 \dd x \right) ^{1/2}.
\]
However, once $n>1$, the first integral is infinite for any $T$, causing this proof to break down. The issue is again a convergence issue; in fact, one can make this proof go through by integrating $1/\norm x_2^r$ for any $r>n$, which means that one can prove a multidimensional uncertainty principle for $M_r(f)$ for any $r>n$. However, we still do not know how to prove the ordinary Heisenberg uncertainty principle in any dimension greater than $1$ using our framework, and it would be very interesting to understand if these convergence issues represent a real limitation of our approach, or if there is a way around them. We leave this tantalizing question as an open problem.
\begin{open}
	Can one prove the multidimensional Heisenberg uncertainty principle, Theorem \ref{thm:multidim-hup}, using our framework? For instance, can one prove that if $g \in L^1(\R^n) \cap L^\infty(\R^n)$, then
	\[
		\frac{\norm g_1}{\norm g_\infty} \leq \left(Cn \frac{V(g)}{\norm g_\infty^2}\right)^c,
	\]
	where $V(g) = \int_{\R^n} \norm x_2^2 \ab{g(x)}^2 \dd x$ is the $n$-dimensional variance of $g$, and $C,c>0$ are constants independent of $n$? Alternately, can one prove such an inequality with $\norm g_\infty$ replaced by $\norm g_q$? In all these questions, the main interest is to obtain the correct dependence on the dimension $n$.
\end{open}

\paragraph{Acknowledgments} We would like to thank Zhengwei Liu and Assaf Naor for helpful discussions, and Tomasz Kosciuszko for correcting an error in an earlier draft of this paper. We would also like to thank Greg Kuperberg for meticulous reading and many insightful comments on an earlier draft, as well as for permission to include his Theorem \ref{thm:kuperberg} in this paper.

\appendix
\section{Proofs of some technical results}\label{sec:appendix}
In this section, we present the proofs of Lemma \ref{lem:non-ab-fourier-props}, Proposition \ref{prop:1-supp-vs-2-supp}, Theorem \ref{thm:our-2-support}, Theorem \ref{thm:incomparable-heisenbergs}, and Theorem \ref{thm:general-heisenberg}, which were omitted from the main text.
\subsection{Proof of Lemma \ref{lem:non-ab-fourier-props}}
\begin{proof}[Proof of Lemma \ref{lem:non-ab-fourier-props}]\hfill
\begin{enumerate}
	\item 
	Recall that the operators $A,B:\C^{n \times n} \to \C^{n \times n}$ are defined by $A\circ M = FMF^*$ and $B\circ N=F^* NF$, and that we showed from Proposition \ref{prop:matrix-entries-orth} that $F^* F = FF^* = nI$. This implies that
	\[
		B\circ (A\circ M) = F^* (F M F^*) F = (nI)M(nI) = n^2 M.
	\]

	\item \label{item:A-bound} Recall that $V$ consists of all matrices $T_f$ for $f \in \C[G]$. To prove this norm bound, it suffices to prove it for the extreme points of the $L^1$ ball. So we may assume that $f=\delta_x$ is the function that takes value $1$ at some $x \in G$ and value $0$ elsewhere. In that case, $T_f$ is a permutation matrix, and therefore every entry of $F T_{\delta_x} F^*$ is the inner product of two columns of $F$. By Proposition \ref{prop:matrix-entries-orth}, all these inner products are either $0$ or $n$, which implies that $\norm{A\circ T_{\delta_x}}_\infty \leq n = \norm{T_{\delta_x}}_1$.
	\item Note that $B=A^*$, and recall that the $L^1$ and $L^\infty$ norms are dual to one another. This implies that $\norm B_{1 \to \infty} = \norm A_{1 \to \infty}$, and thus this is a direct consequence of (\ref{item:A-bound}).
	\qedhere
\end{enumerate}
\end{proof}

\subsection{Proof of Proposition \ref{prop:1-supp-vs-2-supp}}
\begin{proof}[Proof of Proposition \ref{prop:1-supp-vs-2-supp}]
	Assume this were not the case, and let $n$ be the smallest dimension in which a counterexample exists. Consider the set of counterexamples $v$ with $\norm v_1=1$. Since the $L^1$ and $L^2$ norms of a vector are unchanged if we replace each entry by its absolute value and are unchanged if we permute the coordinates, we may restrict ourselves to counterexamples $v$ with non-negative real entries such that $v_1 \geq v_2 \geq \dotsb \geq v_n$. Finally, observe that for any $s$, the set of vectors with $\ab{\supp_{\varepsilon^2}^1(v)} = s$ is a closed set, and similarly for $\ab{\supp_\varepsilon^2(v)}$. So the set of all counterexamples $v$ with $v_1 \geq v_2 \geq \dotsb \geq v_n \geq 0$ and $\norm v_1=1$ is a compact subset of $\R^n$, and we may therefore pick a counterexample $v$ of minimal $L^2$ norm. So from now on, let $v$ be a counterexample with $v_1 \geq \dotsb \geq v_n \geq 0$, $\norm v_1=1$, $n$ chosen to be minimal, and $\norm v_2$ minimal among all such counterexamples. 
	Let $s_1=\ab{\supp_{\varepsilon^2}^1(v)}$ and $s_2 = \ab{\supp_\varepsilon^2(v)}$, so we assume for contradiction that $s_1<s_2$. We split $v$ into three sub-vectors,
	\[
		v_L = (v_1,v_2,\ldots,v_{s_1}) \qquad v_M = (v_{s_1+1},\ldots,v_{s_2-1}) \qquad v_R = (v_{s_2},\ldots,v_n),
	\]
	with the subscripts indicating \emph{Left, Middle,} and \emph{Right}. Note that $v_M$ may be the empty vector if $s_2=s_1+1$. Let the lengths of these vectors be $\ell=s_1,m=s_2-s_1-1$, and $r=n-s_2+1$. We know that $v_L$ contains at least $1- \varepsilon^2$ of the $L^1$ mass of $v$, meaning that
	\[
		\norm{v_L}_1 \geq (1- \varepsilon^2) \norm v_1 = (1- \varepsilon^2)(\norm {v_L}_1+(\norm{v_M}_1+\norm{v_R}_1)),
	\]
	which implies that
	\begin{equation}\label{eq:1-bound}
		\norm{v_L}_1 \geq \frac{1- \varepsilon^2}{\varepsilon^2}(\norm{v_M}_1+\norm{v_R}_1).
	\end{equation}
	Similarly, since $v_L$ and $v_M$ together contain $s_2-1$ coordinates of $v$, they must collectively have \emph{less than} $1- \varepsilon$ of the $L^2$ mass of $v$, meaning that
	\[
		\norm{v_L}_2^2+\norm{v_M}_2^2 < (1- \varepsilon^2)\norm v_2^2 = (1- \varepsilon^2)((\norm {v_L}_2^2+\norm{v_M}_2^2)+\norm{v_R}_2^2),
	\]
	which implies
	\begin{equation}\label{eq:2-bound}
		\norm{v_L}_2^2 + \norm{v_M}_2^2 < \frac{1- \varepsilon^2}{\varepsilon^2} \norm{v_R}_2^2.
	\end{equation}
	Now, we claim that $v_L$ and $v_M$ are both constant vectors. Indeed, suppose not, and let $w$ be the vector gotten by replacing the first $\ell$ entries of $v$ by the average value of $v_1,\ldots,v_{\ell}$, and replacing the next $m$ entries by the average value of $v_{s_1+1},\ldots,v_{s_2-1}$. Then $\norm w_2<\norm v_2$, since the $L^2$-norm is strictly convex, but $\norm w_1 = \norm v_1$. Therefore, inequalities (\ref{eq:1-bound}) and (\ref{eq:2-bound}) both hold for $w$, meaning that $w$ is a new counterexample with smaller $L^2$ norm, which we assumed did not exist. Thus, $v_L$ and $v_M$ are both constant vectors. 
	In other words, we've found that there exist constants $a\geq b \geq 0$ such that $v_L=(a,a,\ldots,a),v_M = (b,b\ldots,b)$, and every entry of $v_R$ is at most $b$. In that case, inequalities (\ref{eq:1-bound}) and (\ref{eq:2-bound}) become
	\begin{align}
		&\ell a \geq \frac{1- \varepsilon^2}{\varepsilon^2}(mb+\norm{v_R}_1)\label{eq:new-1-bound}\\
		&\ell a^2 + mb^2 < \frac{1- \varepsilon^2}{\varepsilon^2} \norm{v_R}_2^2\label{eq:new-2-bound}.
	\end{align}
	Multiplying (\ref{eq:new-1-bound}) by $a$ and using the fact that $m,b \geq 0$, we find that
	\[
		\ell a^2 \geq \frac{1- \varepsilon^2}{\varepsilon^2} a \norm{v_R}_1 \geq \frac{1- \varepsilon^2}{\varepsilon^2} \norm{v_R}_2^2,
	\]
	where the last step uses the fact that every entry of $v_R$ is at most $a$. However, (\ref{eq:new-2-bound}) implies that
	\[
		\ell a^2 < \frac{1- \varepsilon^2}{\varepsilon^2} \norm{v_R}_2^2,
	\]
	a contradiction. 
\end{proof}

\subsection{Proof of Theorem \ref{thm:our-2-support}}
\begin{proof}[Proof of Theorem \ref{thm:our-2-support}]
	Let $U=\frac{1}{\sqrt k}A$ be a rescaling of $A$, chosen so that $U$ is unitary, meaning that $\norm{Uw}_2=\norm w_2$ for all $w \in \C^n$. Note that since the definition of $\supp^2$ is invariant under rescaling, we have that $\ab{\supp_\eta^2(Av)} = \ab{\supp_\eta^2(Uv)}$.

	Let $S=\supp_\varepsilon^2(v)$ and $T= \supp_\eta^2(Uv)$. Let $P_S,P_T$ denote the orthogonal projections onto the coordinates indexed by $S,T$, respectively, and let $M = P_S U^* P_T$ be the submatrix of $U^*$ with rows indexed by $S$ and columns by $T$. Our goal is to obtain upper and lower bounds on $\norm M_{2 \to 2}$, which we will combine to conclude the desired result. To begin with the upper bound, we observe that for any vector $w$,
	\begin{align*}
		\norm{Mw}_2^2 = \sum_{i=1}^n \ab{(Mw)_i}^2 = \sum_{i \in S} \ab{\inn{M_i,w}}^2 \leq \sum_{i \in S} \norm{M_i}^2_2 \norm w_2^2 \leq \frac{\ab S \ab T}{k} \norm w_2^2,
	\end{align*}
	where $M_i$ denotes the $i$th row of $M$. The first inequality is Cauchy--Schwarz, while the second uses the fact that every entry of $M$ has absolute value at most $1/\sqrt k$, and that there are $\ab S \ab T$ entries in $M$. This shows that $\norm M_{2 \to 2} \leq \sqrt{\ab S \ab T/k}$. 

	For the lower bound, we first observe that the unitarity of $U$ implies that
	\[
		\norm{v-U^*P_TUv}_2 = \norm{U^* U v- U^* P_T Uv}_2 = \norm{U^*(Uv-P_T Uv)}_2 = \norm{Uv -P_T Uv}_2 \leq \eta \norm v_2,
	\]
	where the last step follows from the definition of $T$. Since $P_S$ is a projection, it is a contraction in $L^2$, so
	\[
		\norm{P_S v - M U v}_2 = \norm{P_S(v - U^* P_T U v)}_2 \leq \norm {v - U^* P_T Uv}_2 \leq \eta \norm v_2.
	\]
	Moreover, by the definition of $S$, we know that $\norm{v-P_S v}_2 \leq \varepsilon \norm v_2$. Therefore,
	\[
		\norm{v-M U v}_2 \leq \norm{P_S v - M U v}_2 + \norm{v-P_S v}_2 \leq (\eta+\varepsilon) \norm v_2.
	\]
	Using the inequality $\norm{a-b}_2 \geq \norm a_2 - \norm b_2$, we conclude that
	\[
		\norm{M U v}_2 \geq (1- \varepsilon - \eta) \norm v_2 = (1- \varepsilon - \eta) \norm{Uv}_2.
	\]
	Combining this with our bound $\norm M_{2 \to 2} \leq \sqrt{\ab S \ab T/k}$ gives the desired result. 
\end{proof}

\subsection{Proof of Theorem \ref{thm:incomparable-heisenbergs}}
\begin{proof}[Proof of Theorem \ref{thm:incomparable-heisenbergs}]
	Let $a>b>0$ be real numbers. Define
	\[
		f_{a,b}(x) = e^{-\pi ((a+bi)x)^2} = e^{-\pi (a^2-b^2)x^2} e^{-2\pi i ab x^2}.
	\]
	From the definition, we see that $\ab{f_{a,b}(x)} = e^{-\pi(a^2-b^2)x^2}$. Since we assumed that $a>b>0$, we have that $a^2-b^2>0$, and therefore $\ab{f_{a,b}}$ decays superexponentially at infinity, and we see that $f_{a,b} \in \sch(\R)$. We can compute
	\[
		\norm{f_{a,b}}_2^2 = \int \ab{f_{a,b}(x)}^2\dd x = \int e^{-2\pi (a^2-b^2)x^2} = \frac{1}{\sqrt{2(a^2-b^2)}}.
	\]
	Additionally, for any fixed $a>b>0$, we see that $\pi(a^2-b^2)x^2$ is minimized at $x=0$, implying that $\ab{f_{a,b}(x)}$ is maximized at $x=0$, and therefore
	\[
		\norm{f_{a,b}}_\infty = \ab{f_{a,b}(0)} = 1.
	\]
	We can also compute the Fourier transform of $f_{a,b}$ explicitly, by recalling that the Fourier transform of $e^{-\pi(c x)^2}$ is $\frac 1c e^{-\pi(\xi/c)^2}$, and that this holds for all $c \in \C$ for which $e^{-\pi(c x)^2} \in L^2$. Setting $c=a+bi$, we find that
	\[
		\wh{f_{a,b}}(\xi) = \frac{1}{a+bi} e^{-\pi \left( \frac{\xi}{a+bi} \right) ^2} = \frac{1}{a+bi} e^{-\pi \xi^2 (a^2-b^2)/(a^2+b^2)^2} e^{-2\pi i \xi^2 ab/(a^2+b^2)^2}.
	\]
	In particular,
	\[
		\bab{\wh{f_{a,b}}(\xi)} = \frac{1}{\sqrt{a^2+b^2}}e^{-\pi \xi^2 (a^2-b^2)/(a^2+b^2)^2}.
	\]
	Again for fixed $a>b>0$, we have that $\ab{\wh{f_{a,b}}(\xi)}$ is maximized at $\xi=0$, and we conclude that
	\[
		\norm{\wh{f_{a,b}}}_\infty = \frac{1}{\sqrt{a^2+b^2}}.
	\]
	This implies that 
	\[
		F(f_{a,b}) = \frac{\norm{f_{a,b}}_\infty \norm{\wh{f_{a,b}}}_\infty}{\norm{f_{a,b}}_2^2}=\sqrt{\frac{2(a^2-b^2)}{a^2+b^2}}.
	\]
	Finally, let us set $b =\sqrt{a^2-1}$, and insist that $a>1$ so that $b>0$. Then we get that
	\[
		F(f_{a,\sqrt{a^2-1}}) = \sqrt{\frac{2}{2a^2-1}},
	\]
	and as $a$ ranges over $(1,\infty)$, the function $\sqrt{2/(2a^2-1)}$ ranges over $(0,\sqrt 2)$. So we conclude that $(0,\sqrt 2)$ is in the image of $F$.

	Next, we wish to construct a family of functions whose images under $F$ cover the remaining interval $[\sqrt 2,\infty)$. For a real number $c>0$, we define
	\[
		g_c(x) = \frac{1}{\sqrt c} e^{-\pi(x/c)^2} + \sqrt c e^{-\pi(cx)^2}.
	\]
	From the property mentioned above about the Fourier transform of a Gaussian, we see that $g_c$ is its own Fourier transform for all $c$. We can compute
	\[
		\norm{\wh{g_c}}_\infty =\norm{g_c}_\infty = \ab{g_c(0)} = \sqrt c+ \frac{1}{\sqrt c}.
	\]
	Moreover, we can also compute
	\begin{align*}
		\norm{g_c}_2^2 &= \int_{-\infty}^\infty \left( \frac{1}{\sqrt c} e^{-\pi(x/c)^2} + \sqrt c e^{-\pi(cx)^2} \right) ^2 \dd x\\
		&=\frac 1c \int_{-\infty}^\infty e^{-2\pi(x/c)^2} \dd x+ 2 \int_{-\infty}^\infty e^{-\pi x^2(c^2+1/c^2)} \dd x+c \int_{-\infty}^\infty e^{-2\pi(cx)^2}\dd x\\
		&=\frac 1c \left( \frac{c}{\sqrt 2} \right) +2 \left( \frac{1}{\sqrt{c^2+\frac{1}{c^2}}} \right) + c \left( \frac{1}{c\sqrt 2} \right) \\
		&=\sqrt 2 + \frac{2c}{\sqrt{c^4+1}}.
	\end{align*}
	Thus, we find that
	\begin{align*}
		F(g_c)&=\frac{\norm{g_c}_\infty \norm{\wh{g_c}}_\infty}{\norm{g_c}^2_2}=\frac{(\sqrt c+\frac{1}{\sqrt c})^2}{\sqrt 2 + \frac{2c}{\sqrt{c^4+1}}}=\frac{c+2+\frac 1c}{\sqrt 2 + \frac{2c}{\sqrt{c^4+1}}}.
	\end{align*}
	This function is minimized at $c=1$, where its value is $\sqrt 2$. Moreover, as $c \to \infty$, the denominator converges to $\sqrt 2$, whereas the numerator grows like $c$. Thus, we see that $\lim_{c \to \infty}F(g_c)=\infty$. Since this is a continuous function of $c$, we conclude that $[\sqrt 2,\infty)$ is in the image of $F$. Combining this with our result that $(0,\sqrt 2)$ is in the image of $F$, we conclude that the image of $F$ is all of $\R_{>0}$.
\end{proof}

\subsection{Proof of Theorem \ref{thm:general-heisenberg}}
\begin{proof}[Proof of Theorem \ref{thm:general-heisenberg}]
	We may assume that $f \neq 0$. We mimic the proof of Theorem \ref{thm:our-heisenberg}. We first claim that for any non-zero $g \in L^1(\R) \cap L^\infty(\R)$,
	\begin{equation}\label{eq:mr-bound}
		\frac{\norm g_1}{\norm g_q} \leq \frac{1}{C_{r,q}} \left(  \frac{M_r(g)}{\norm g_q^2} \right) ^{\frac{q-1}{qr+q-2}},
	\end{equation}
	where
	\[
		C_{r,q} = (r-1)^{\frac{q-1}{qr+q-2}} 2^{-\frac{2qr+q-r-2}{qr+q-2}}
	\]
	Once we have this, we can apply it with the norm uncertainty inequality, Theorem \ref{thm:expanding-q-norm}, to obtain that
	\[
		\frac{1}{C_{r,q}C_{s,q}}\left(  \frac{M_r(f)}{\norm f_q^2} \right) ^{\frac{q-1}{qr+q-2}} \left(  \frac{M_s({Af})}{\norm {Af}_q^2} \right) ^{\frac{q-1}{qs+q-2}} \geq \frac{\norm f_1}{\norm f_q} \cdot \frac{\norm{Af}_1}{\norm{Af}_1} \geq k^{\frac{q-1}{q}},
	\]
	which is the claimed result. So it suffices to prove (\ref{eq:mr-bound}).

	As before, we set $T=\frac 12 (\norm g_1/(2\norm g_q))^{q/(q-1)}$, so that $(2T)^{1-1/q} \norm g_q =\frac 12 \norm g_1$. H\"older's inequality again gives that
	\[
		\int_{-T}^T \ab{g(x)}\dd x \leq \frac 12 \norm g_1, \qquad \text{implying} \qquad \int_{\ab x>T} \ab{g(x)}\dd x \geq \frac 12 \norm g_1.
	\]
	Therefore, applying the Cauchy--Schwarz inequality, we can compute
	\begin{align*}
		\frac 12 \norm g_1 &\leq \int_{\ab x >T} \ab{g(x)} \dd x\\
		&=\int_{\ab x>T} \frac{1}{\ab x^{r/2}} \left(\ab x^{r/2} \ab{g(x)}\right) \dd x\\
		&\leq \left( \int_{\ab x>T} \frac{1}{\ab x^r}\dd x \right) ^{1/2} \left( \int_{\ab x>T} \ab x^r \ab{g(x)}^2\dd x \right) ^{1/2}\\
		&\leq \frac{\sqrt 2}{\sqrt{(r-1)T^{r-1}}} M_r(g) ^{1/2},
	\end{align*}
	where we use the assumption that $r>1$ to evaluate the (convergent) integral. Rearranging, and using the definition of $T$, we obtain (\ref{eq:mr-bound}).
\end{proof}

\end{document}